\documentclass[12pt]{amsart}
\usepackage{amsfonts,amsmath,amssymb,amscd}

\usepackage{color}

\usepackage{bbm}
\usepackage{algorithm}
\usepackage{algpseudocode}

\usepackage[backref=page]{hyperref}
\hypersetup{
  colorlinks=true, 
   linkcolor=red,  
}

\usepackage[margin=1in]{geometry}

\begin{document}

\newcommand{\eq}{\begin{equation}}
\newcommand{\en}{\end{equation}}
\def\eqa{\begin{eqnarray}}
\def\ena{\end{eqnarray}}
\newcommand{\ignore}[1]{}
\def\no{\noindent}
\def\e{\mathbb{E \,}}
\def\M{{\mathcal M}}
\def\BN{\mathbb{N}}
\def\BR{\mathbb{R}}
\def\L{{\mathcal L}}
\def\p{\mathbb{P}}
\def\var{{\rm Var}}
\def\SD{{\rm SD}}
\def\A{{\mathcal A}}
\def\B{{\mathcal B}}
\def\C{{\mathcal C}}
\def\F{{\mathcal F}}
\def\X{{\mathcal X}}

\def\ta{{\alpha}}
\def\tb{{\beta}}
\def\tc{{\gamma}}

\def\PD{{\mathcal PD}}
\def\Q{{\mathcal Q}}
\def\R{{\mathcal R}}
\def\X{{\mathcal X}}
\def\hcX{\hat{\mathcal{X}}}
\def\hX{\hat{{X}}}
\def\Y{{\mathcal Y}}
\def\BZ{\mathbb{Z}}
\def\UZ{{\mathbf{Z}}}

\def\ra{\rightarrow}
\def\bone{{\mathchoice {\rm 1\mskip-4mu l} {\rm 1\mskip-4mu l}
{\rm 1\mskip-4.5mu l} {\rm 1\mskip-5mu l}}}

\newtheorem{theorem}{Theorem}
\newtheorem{lemma}[theorem]{Lemma}

\theoremstyle{definition}
\newtheorem{remark}[theorem]{Remark}
\newtheorem{example}[theorem]{Example}
  \numberwithin{theorem}{section}

\setcounter{tocdepth}{2}
 
\title[Probabilistic divide-and-conquer]{ Probabilistic divide-and-conquer:  a new  exact simulation method, with integer partitions as an example }

\author[Arratia]{Richard Arratia}
\address[Richard Arratia]{Department of Mathematics, University of Southern California,
Los Angeles CA 90089.}
\email{rarratia@math.usc.edu}

\author[DeSalvo]{Stephen DeSalvo}
\address[ Stephen DeSalvo]{Department of Mathematics, University of California Los Ageles,
Los Angeles CA 90024}
\email{stephendesalvo@math.ucla.edu}
 
\date{November 23, 2015}

\begin{abstract}
We propose a new method, probabilistic divide-and-conquer, for improving the success probability in rejection sampling.  
For the example
of integer partitions, there is an ideal recursive scheme which improves the rejection cost from asymptotically  order 
$n^{3/4}$  to a constant.  We show other examples for which a non-recursive, one-time application of probabilistic divide-and-conquer removes a substantial fraction of the rejection sampling cost.  

We also present a variation of probabilistic divide-and-conquer for generating i.i.d.~samples that exploits features of the coupon collector's problem, in order to obtain a cost that is sublinear in the number of samples.
\end{abstract}
 
\maketitle

\tableofcontents

\section{Introduction.}
   
\subsection{Exact simulation}

The task of 
simulation is to provide 
one or more
samples from a set 
according to some given probability distribution.   For many 
combinatorial problems, the given distribution is the uniform choice over all
possibilities of a fixed size.  
A host of combinatorial objects, including assemblies, multisets, and selections, may be expressed, see 
\cite{IPARCS}, in terms of
a process of independent random variables, conditional on a weighted sum --- the object size --- equalling a given target.  

The ``Boltzmann sampler," see \cite{Boltzmann}, samples this independent process once and is content with an object size that is \emph{close} to the given target.  
The \emph{rejection sampling} algorithm, on the other hand, 
samples this independent process \emph{repeatedly} until the condition is satisfied.  
The main measure of the efficiency of such a rejection sampling algorithm is the expected number of times the independent process must be sampled before the condition is satisfied.


\ignore{
 Divide-and-conquer is a basic strategy for algorithms.  
Notable examples include the Cooley-Tukey fast Fourier transform, attributable to Gauss \cite{gauss}, and Karatsuba's fast multiplication algorithm \cite{karatsuba}, which surprised Kolmogorov, \cite{karatsuba2}.  We note that these and 
other cases treated in textbooks on algorithms are \emph{deterministic}.  Randomized quicksort, see for example
\cite[Section 7.3]{Cormen2001}, is the prototype of divide-and-conquer using randomness, but such algorithms can be thought of as a variation on the deterministic algorithm, applied to a permutation of the input data.
}

The essence of \emph{probabilistic divide-and-conquer}  --- PDC  --- is random sampling of conditional distributions, which, \emph{when appropriately pieced together}, represent a sample from the target distribution.    
We assume throughout that the target object $S$ may be expressed as a pair $(A,B)$, where
\begin{equation}\label{def AB}
  A \in \A, \ B \in \B \ \mbox{ have given distributions},
\end{equation}
\begin{equation}\label{indep AB}
  A , B \ \mbox{ are independent},
\end{equation}
\begin{equation}\label{def h}
h: \A \times \B \to \{0,1 \} 
\end{equation}
$$
 \mbox{ satisfies   }  p := \e h(A,B) \in (0,1],\footnote{
 The requirement that $p>0$ is  a choice we make here, for the sake of simpler exposition;  there are $p=0$ examples where divide-and-conquer is useful, see \cite{PDCDSH}.
In cases where $p=0$,
 the conditional distribution, apparently specified by \eqref{def S}, needs further specification --- this is known as
 Borel's paradox.
 }
 $$
where, of course, we also assume that $h$ is measurable, and
\begin{equation}\label{def S}
S \in \A \times \B  \mbox{ has distribution  }  \L(S) = \L(\,  (A,B)\, | \,h(A,B) =1),
\end{equation}
i.e., the law of $S$ is the law of the independent pair $(A,B)$ \emph{conditional on} 
having $h(A,B)=1$.

Then 
rejection sampling may be viewed as sampling $(A_1,B_1),
(A_2,B_2),\ldots$ from the law of $(A,B)$ until $h(A_i,B_i)=1$.  PDC can be described as sampling from the law of $(A\, |\, h(A,B)=1)$ \emph{first}, say with observation $x$, and \emph{then} sampling from the law of $(B\, |\, h(x,B)=1)$.  A simple comparison of algorithms is below.

\begin{algorithm}{\rm
\begin{algorithmic}
\State 1.  Generate sample from $\L(A)$, call it $a$.
\State 2.  Generate sample from $\L(B)$, call it $b$.
\State 3.  Check if $h(a,b)=1$; if so, return $(a,b)$, otherwise restart.
\end{algorithmic}}
\caption{Rejection Sampling}
\label{WTGL procedure}
\end{algorithm}

\begin{algorithm}{\rm
\begin{algorithmic}
\State 1. Generate sample from $\L(A\, |\, h(A,B) = 1),$ call it $x$.
\State 2. Generate sample from $\L(B\, |\, h(x,B) = 1),$ call it $y$.
\State 3. Return $(x,y)$.
\end{algorithmic}}
\caption{Probabilistic Divide-and-Conquer} 
\label{PDC procedure}
\end{algorithm}

Our main tool is von Neumann's acceptance/rejection method, which is not necessary for PDC, but which provides a simple and effective means for sampling from the conditional distributions in Algorithm~\ref{PDC procedure}.  We review this method in Section~\ref{sect acceptance}.  
Our primary application of Algorithm 2 with von Neumann's rejection method is Algorithm~\ref{PDC procedure von Neumann}, where the rejection cost is split between Step~2 for sampling $A$, and Step~3 for sampling $B$.  Furthemore,
Step~3 in Algorithm~\ref{PDC procedure von Neumann} can be performed either directly or recursively; we give examples of each in Section~\ref{Algorithms Section}.

\begin{algorithm}{\rm
\begin{algorithmic}
\State 1. Generate sample from $\L(A),$ call it $a$.
\State 2. Accept $a$ with probability $t(a)$, where $t(a)$ is a function of $\L(B)$ \\
\ \ \ \ and $h$; otherwise, restart.
\State 3. Generate sample from $\L(B\, |\, h(a,B) = 1),$ call it $y$.
\State 4. Return $(a,y)$.
\end{algorithmic}}
\caption{Probabilistic Divide-and-Conquer via von Neumann} 
\label{PDC procedure von Neumann}
\end{algorithm}

A key class of examples is discussed in Section \ref{sect b=1}, 
where the calculation for $t(a)$ is very simple, and the speedup is order of $\sqrt{n}$ for the specific examples relating to $k$-cores and set partitions, and of order $n^{1/3}$ for the specific example relating to plane partitions.
In these examples, the PDC is not recursive 
and the programming is very easy. 
At the opposite extreme, when 
$B$ can be made into a scaled replica of the original problem, and calculation of $t(a)$ is still
easy enough  --- as in the example of ordinary integer partitions, thanks to results of Hardy and Ramanujan \cite{HR18}, Rademacher \cite{Rademacher} and  Lehmer \cite{Lehmer38,Lehmer39} --- a recursive, self-similar PDC is available, and although the algorithm is more involved, this gives the fastest method for large $n$.


One progenitor to PDC is by McKay and Wormald \cite{WormaldMcKay}, where the ability to calculate likelihood ratios is combined with 
acceptance/rejection sampling to reduce the number of loops and double edges implied by a random configuration;  the luck required to get a simple graph is spread over several stages.  
Another progenitor to PDC, this time in the recursive, self-similar case, is by Alonso~\cite{Alonso} on sampling of Motzkin words. 
In that case, the rejection probabilities of partially completed Motzkin words are given explicitly by quotients of binomial coefficients.

A separate application  
of our divide-and-conquer paradigm, championed in Section~\ref{mam}, is what we call \emph{mix-and-match}.  
Our desire is to combine $A_1,A_2,\ldots,A_m$, i.i.d.~samples from $\L(A)$, and $B_1,B_2,\ldots,B_m$, i.i.d.~samples from $\L(B)$, with $m^2$ chances to have $h(A_i,B_j)=1$.  However, doing this blindly does not
achieve the goal of exact sampling, although it does yield a \emph{consistent} estimator; see~\cite[Section 4.6]{DeSalvoThesis}. 
A particularly practical variation, dubbed \emph{deliberately missing data}~\cite[Section 4.4.3]{DeSalvoThesis}, fixes some number $v$ of samples in advance, and samples $A_1, A_2, \ldots$ and $B_1, B_2, \ldots$ until exactly $v$ matching pairs have been found; this procedure is notably sublinear in $v$. 
Such a comparison of two lists has been exploited previously, as in \cite{DiffieHellman}, the meet-in-the-middle attack, and as in
biological sequence matching, \cite{ArratiaWaterman1,ArratiaWaterman2}, where  for two independent sequences of i.i.d.~letters, we are interested in finding contiguous blocks of letters that appear in both sequences.  

\ignore{
A separate application  
of our divide-and-conquer paradigm, introduced in~\cite{DeSalvoThesis}, is what we call \emph{mix-and-match}.  
Our desire is to combine $A_1,A_2,\ldots,A_m$, i.i.d.~samples from $\L(A)$, and $B_1,B_2,\ldots,B_m$, i.i.d.~samples from $\L(B)$, with $m^2$ chances to have $h(A_i,B_j)=1$.  However, doing this blindly does not
achieve the goal of exact sampling, although it does yield a \emph{consistent} estimator. 
Such a comparison of two lists has been exploited previously, as in \cite{DiffieHellman}, the meet-in-the-middle attack, and as in
biological sequence matching, \cite{ArratiaWaterman1,ArratiaWaterman2}, where  for two independent sequences of i.i.d.~letters, we are interested in finding contiguous blocks of letters that appear in both sequences.  
}
\ignore{
A particularly practical variation of mix-and-match samples $X_1, X_2, \ldots, X_m$, i.i.d.~samples from $\L(A | h(A,B)=1),$ followed by sampling of $B_1, B_2, \ldots$ sequentially, such that the first matching pair which satisfies $h(X_i, B_j) = 1,$ where $i$ and $j$ can be taken as the smallest index of all possible matches, is used as the $i^{th}$ element of the sample\footnote{Note that using this as the $j^{th}$ element of the sample introduces bias.}.  We can continue this until all $X_i$'s have a match, $i=1,2,\ldots, m,$ which is a generalization to the coupon collector's problem; alternatively, a variation called \emph{deliberately missing data} only continues generating $B_j$'s until exactly $v \leq m$ of the $X_i$'s have a matching pair, and leaves the remaining unpaired. 
The advantage of this approach is that it is sublinear in the number of samples from the distribution which generates $B_j$'s. 
While we do not discuss this approach further, we refer the interested reader to~\cite{DeSalvoThesis}. 
}

Finally, we present in Section~\ref{appendix} practical considerations when implementing PDC for integer partitions on a computer.  

\ignore{
There is another novel feature to sampling using mix-and-match.
  In mix-and-match, to get a sample of size $r$,  we gain not only from the coupon collector's advantage  $(\log r) / r$, but also, in cases where multiple copies $k$ of the same colors are in the list of required tasks, passed from the $A$ phase to the $B$ phase, by the multiple coupon collector's advantage, involving $\log \log k$  \cite{DoubleDixieCup}.  Of course, in mix-and-match the B distribution on color probabilities is not uniform, and the numbers $k_c$ needed of color $c$ varies over the colors, so no simple analysis of the speedup seems evident.
   }
   
\subsection{An example: integer partitions}
Our first example to illustrate the use of PDC is uniform sampling of integer partitions of a given size~$n$.  Additional examples are given in Section~\ref{additional examples}.  
The starting point is 
 a conditional distribution of the form 
\[\left((Z_1, Z_2, \ldots, Z_n) \bigg| \sum_{i=1}^n i\, Z_i = n\right),
\]
where $Z_1, \ldots, Z_n$ are \emph{independent} geometric distributions with respective parameters $p_1, \ldots, p_n$, with $p_i = 1 - e^{-i \pi/\sqrt{6n}}$, $i=1,\ldots,n$.
The precise description of how they correspond to a random integer partition is described in Section~\ref{sect grand canonical}; for now, we simply wish to emphasize the simple and explicit nature of the following algorithms, in a way accessible to the reader solely interested in implementing such algorithms on a computer.  Many combinatorial structures follow a similar pattern; see \cite{IPARCS}.

We now summarize the relevant rejection sampling and PDC algorithms, and postpone their full explanations until Section~\ref{Algorithms Section}.  One simple measure of an algorithm's performance is the asymptotically expected number of times the $n$--tuple $(Z_1, \ldots, Z_n)$ must be sampled before the simulation terminates and returns a valid sample, which we call the asymptotic \# rejections in the table below.  The following quantities are proved in Section~\ref{Algorithms Section}. 
\[
\begin{array}{lc}
\text{Algorithm} & \text{Asymptotic \# rejections} \\ \hline
\text{Algorithm~\ref{IP WTGL procedure}: Rejection Sampling} & 2 \sqrt[4]{6}\, n^{3/4} \\
\text{Algorithm~\ref{IP PDC DSH procedure}: PDC Deterministic} & 2 \, \pi \sqrt[4]{6}\, n^{1/4} \\
\text{Algorithm~\ref{IP SS PDC procedure}: Self--Similar PDC} & 2\sqrt{2}.
\end{array}
\]

In what follows, $U$ will denote a uniform random variable in the interval $(0,1),$ independent of all other random variables, and  for use in Algorithm~\ref{IP SS PDC procedure} we define, via \eqref{dist Z} and \eqref{def x(n)},
\begin{equation}\label{duck not 2}
f_n(j) := \p_{x(n)}\left(\sum_{i\le n/2} 2i\, Z_{2i} = 2 j\right). \end{equation}
\begin{algorithm}{\rm
\begin{algorithmic}
\State 1.  Generate sample from $\L(Z_1, Z_2, \ldots, Z_n)$, call it $(z_1, \ldots, z_n)$.
\State 2.  If $\sum_{i=1}^n i\, z_i = n,$ return $(z_1, \ldots, z_n)$; else restart.
\end{algorithmic}}
\caption{Rejection Sampling of Integer Partitions}
\label{IP WTGL procedure}
\end{algorithm}

\begin{algorithm}{\rm
\begin{algorithmic}
\State 1. Generate sample from $\L(Z_2, \ldots, Z_n),$ call it $(z_2, \ldots, z_n)$.  \\
\ \ \ \ Set $k := n - \sum_{i=2}^n i\, z_i$.
\State 2. If $k \geq 0$ and $U < e^{-k\,\pi/\sqrt{6n}}$, return $(k, z_2, \ldots, z_n)$; else restart.
\end{algorithmic}}
\caption{Probabilistic Divide-and-Conquer Deterministic Second Half for Integer Partitions} 
\label{IP PDC DSH procedure}
\end{algorithm}

\begin{algorithm}{\rm
\begin{algorithmic}
\Procedure{SS\_PDC\_IP}{$n$}
\State 0.  If $n=1$, return 1; otherwise,
\State 1. Generate sample from $\L(Z_1, Z_3, Z_5, \ldots),$ call it $(z_1, z_3, z_5, \ldots)$.
\State 2. Set $k := n - \sum_{i\, odd} i\, z_i$.  
\State 3. If $k<0$ or $k$ is odd, restart.

\State 4. If $U < f_n(k/2) / \max_\ell f_n(\ell),$\\
\qquad \qquad let $(z_2, z_4, \ldots) =\ $SS\_PDC\_IP($k/2$);\\ 
\qquad \qquad return $(z_1, z_2, \ldots, z_n,0,0,\ldots).$
\State \quad\  Else restart.
\EndProcedure
\end{algorithmic}}
\caption{Self--Similar Probabilistic Divide-and-Conquer for Integer Partitions} 
\label{IP SS PDC procedure}
\end{algorithm}

\begin{remark}\label{X}
Note that while Algorithms~\ref{IP WTGL procedure}~and~\ref{IP PDC DSH procedure} are entirely explicit, Algorithm~\ref{IP SS PDC procedure} relies on our ability to determine whether $U < f_n(k/2)/\max_\ell f_n(\ell)$.  \emph{This is not the same as being able to compute $f_n(j)$ for various values of the parameters.}  The task in Algorithm~\ref{IP SS PDC procedure} is considerably easier, since we simply need to compare the leading bits of each quantity until there is disagreement.  In fact, \emph{on average we only need the leading 2 bits of each side of the inequality in order to determine the rejection in Step 4 of Algorithm~\ref{IP SS PDC procedure}}; see \cite{KnuthYao} and our discussion in Section~\ref{floating}.    
  There are also several other improvements to Algorithm~\ref{IP SS PDC procedure} that can be applied at the expense of further complicating the algorithm; these are discussed in Section~\ref{sect odd}.
\end{remark}

\ignore{

We focus on three  versions of the general PDC idea applied to integer partitions.  
Each method expresses a partition as $\lambda =(A,B)$.
\begin{enumerate}

\item In $(A,B)$, $A$ is the (list of) large parts, say $\lfloor \sqrt{n}\rfloor +1$ up to $n$, $B$ is the small parts, size 1 up to $\lfloor \sqrt{n} \rfloor$. Mix-and-match  offers an additional speedup, but analysis of the cost is not easy.

\item In $(A,B)$, $A$ specifies the number of parts of sizes 2 through $n$, $B$ is the number of parts of size 1.  Hence the $B$ side of the simulation is trivial, with no calls to a random number generator.  Nevertheless, there is a large speedup, by a factor asymptotic to  
$\sqrt{n}/c$, where $c  = \pi/\sqrt{6}$.  This method is remarkably robust and powerful; see section
\ref{sect b=1}.

\item  In $(A,B)$, $A$ corresponds to $d(z)=\prod_{i \ge 1} (1+z^i)$ enumerating partitions with distinct parts, so that we may exploit the classic identity\footnote{where, of course, $p(z)=\prod_{i \ge 1} (1-z^i)^{-1}$ enumerates all partitions.}
 $$p(z) =d(z) p(z^2). $$    This method iterates beautifully, reducing the target, $n$, by a factor of approximately 4 per iteration, with an acceptance/rejection cost of only roughly $2 \sqrt{2}$, improved  in Section \ref{sect parity} to $\sqrt{2}$, by combining with the idea in method (2).  \label{hey 3}
\end{enumerate}
}
 
On a personal computer, the RandomPartition function in Mathematica\textsuperscript{\textregistered}'s Combinatorica package \cite{Mathematica}
appears to hit the wall at around $n=2^{20}$.  On the same computer, the recursive divide and conquer algorithm in Algorithm~\ref{IP SS PDC procedure}
can handle $n$ as large\footnote{In Matlab \cite{MATLAB}, we got up to $2^{49}$; the $2^{58}$ is from a C++
implementation;  in both cases, memory rather than time is the limiting factor.}
 as $2^{58}$.  
Relative to the rejection sampling algorithm in Algorithm~\ref{IP WTGL procedure}, analyzed in Section \ref{sect grand canonical},
the recursive  divide-and-conquer achieves more than a trillion-fold speedup at $n=2^{58}$.

\section{The basic lemma  for exact  sampling  with divide-and-conquer}
\ignore{
We assume throughout that
\begin{equation}\label{def AB}
  A \in \A, \ B \in \B \ \mbox{ have given distributions},
\end{equation}
\begin{equation}\label{indep AB}
  A , B \ \mbox{ are independent},
\end{equation}
\begin{equation}\label{def h}
h: \A \times \B \to \{0,1 \} 
\end{equation}
$$
 \mbox{ satisfies   }  p := \e h(A,B) \in (0,1],\footnote{
 The requirement that $p>0$ is \emph{not needed} for divide-and-conquer to be useful;
 but rather, a choice we make for the sake of simpler exposition.  In cases where $p=0$,
 the conditional distribution, apparently specified by \eqref{def S}, needs further specification --- this is known as
 Borel's paradox.
 }
 $$
where, of course, we also assume that $h$ is measurable, and
\begin{equation}\label{def S}
S \in \A \times \B  \mbox{ has distribution  }  \L(S) = \L(\,  (A,B)\, | \,h(A,B) =1),
\end{equation}
i.e., the law of $S$ is the law of the independent pair $(A,B)$ \emph{conditional on} 
having $h(A,B)=1$.
} 
\ignore{
Note that a restatement of \eqref{def S}, exploiting the hypothesis \eqref{def h}, is that
for measurable sets $R \subset \A \times \B$, 
$$
\p(S \in R ) = \frac{\p((A,B) \in R \mbox{ and } h(A,B)=1)}{p},
$$
or equivalently, for bounded measurable functions $g$ from $\A \times \B$ to the real numbers, 
$$
 \e g(S) =  \e( g(A,B) h(A,B)) \, /p.
 $$

Since we have assumed $p>0$, this is elementary conditioning.  This allows   the distributions of $A$ and $B$ to be arbitrary:  discrete,   absolutely continuous, or otherwise.
}
\subsection{Statement of the PDC Lemma}
The following lemma is a straightforward application of Bayes' formula.

\begin{lemma}\label{lemma 1}
Suppose $X$ is a random element of $\A$ with distribution 
\begin{equation}\label{def X}
   \L(X) = \L( \, A \, | \, h(A,B)=1 \, ),
\end{equation}
and  $Y$ is a random element of $\B$ with conditional distribution
\begin{equation}\label{def Y}
   \L(Y \, | X=a ) = \L( \, B \, | \, h(a,B)=1 \, ).
\end{equation}
Then $(X,Y) =^d S$, i.e., the pair $(X,Y)$ has the same distribution  as $S$, given by \eqref{def S}.
\end{lemma}
\ignore{
\begin{proof}
A restatement of \eqref{def X} is
\begin{equation}\label{A to X}
\p(X \in da) = \frac{\p(A \in da) \, \e h(a,B)}{p},
\end{equation}
and relation \eqref{def Y}
is equivalent to
\begin{equation}\label{def Y 2}
   \L((X,Y) \, | X=a ) = \L( \, (a,B) \, | \, h(a,B)=1 \, ).
\end{equation}
Hence, for any bounded measurable $g: \A \times \B \to \BR$,
\begin{eqnarray*}
  \e g(X,Y) & = & \e( \e( g(X,Y)|X)) \\
     & = & \int_\A \p (X \in da) \ \ \ \ \e( g(X,Y) | X=a) \\
     & = & \int_\A \frac{\p(A \in da) \, \e h(a,B)}{p} \ \frac{ \e( g(a,B) h(a,B))}{\e h(a,B)}\\
    & = &\frac{1}{p}  \int_\A \p(A \in da)  \ \e( g(a,B) h(a,B)) \\
   & = & \e (g(A,B) | h(A,B)=1) = \e g(S).
 \end{eqnarray*}
 We used \eqref{def Y 2} for the middle line in the display above;  on the set $\A_0 := \{a \in \A: \e h(a,B)=0 \}$,
 which contributes 0 to the integral, we took the usual liberty of dividing by 0, rather than writing out separate expressions for the integrals over $\A_0$ and $\A \setminus \A_0$.
\end{proof}
}

\subsection{Use of acceptance/rejection sampling}\label{sect acceptance}
Assume that we know how to generate a sample from $\L(A)$ --- this is under the distribution in \eqref{def AB}, where $A,B$ are independent.
But, we need instead to sample from an alternate distribution, $\L(A\, |\, h(A,B)=1)$, denoted above as that of $X \in \A$. 
The rejection method recipe, for using Equation~\eqref{def X}, may be viewed as having 2 steps, as follows. 
Let $p_a := \e h(a,B)$.  

\begin{enumerate}

\item[(a)] Find a threshold function $t: \A \to [0,1]$, with $t(a)$ proportional to $p_a,$ and $t(a) \leq 1,$ for all $a \in \A$. 

\item[(b)] Propose i.i.d.~samples $A_1,A_2,\ldots$ and independently generate uniform (0,1) random variables $U_1,U_2,\ldots$, until $U_i \leq t(A_i).$

\end{enumerate}
Assuming that we can find an $a^*$ where $p_a$ achieves its maximum value, then the optimal function is $t(a) = p_a / p_{a^*}$. 
However, we note that even when $p_{a^*}$ is difficult to compute exactly, we simply need an upper bound on $p_{a^*}$, and really only the ability to compute the leading bits of $t(a)$ as needed until a disagreement is reached between the bits of a uniform random number in $[0,1]$ and $t(a)$; see Remark~\ref{X} and Section~\ref{sect cost}. 

For comparison with hard rejection sampling, we write 
\[ t(a) = \frac{C}{p}\,  p_a, \] 
where $C$ is allowed to be \emph{any} positive constant such that $C \leq p/p_{a^*}$, with optimal value given by $C = p / p_{a^*}$. 
See \cite{DeVroye} for a thorough treatment of acceptance/rejection methods.

The expected fraction of proposed samples $A_i$ to be accepted will be the average of $t(a)$ with respect to the distribution of $A$, i.e.,
$$
 p_{\rm acc} := \p( U \le t(A)) = \e t(A)  = C  \times \frac{p_A}{p} = C,
$$
and the expected number of proposals needed to get each acceptance is the reciprocal of this, so we define
\begin{equation}\label{acceptance cost}
\mbox{ Acceptance cost} := \frac{1}{p_{\rm acc}} =\frac{1}{C} = \frac{p_{a^*}}{p},
\end{equation}
where the last equality assumes assumes the optimal choice of $C$.
\ignore{ 
\begin{remark}
A priori, it appears as though the threshold function $t(a)$ depends on $p$; however, the scaling by $C$ eliminates this dependency, and simply requires detailed knowledge of the distribution of $B$.  It would be tempting to replace $C$ by $C' = C/p$ in our exposition, and define $t(a)$ as
\[ t(a) = C' \times \e h(a,B) = \frac{\e h(a,B)}{\e h(a^\ast,B)}, \]
but since the overall cost of the algorithm is defined in terms of $C$, and since we will be comparing this approach to the usual rejection sampling, which has overall cost $1/p$, we prefer the separate notation.
\end{remark}
} 
For comparison, if we were using rejection sampling instead of divide-and-conquer, i.e., proposing pairs $(A_i,B_i)$ until $h(A_i,B_i)=1$, with success probability $p$ at each trial and expected number of proposals to get one success equal to $1/p$, then, \emph{ignoring the cost of proposing the $B_i$}, the 
ratio of old cost to new might be called a speedup:
\begin{equation}\label{speedup}
\text{speedup} = \frac{\text{old cost}}{\text{new cost}} = \frac{1/p}{1/C} = \frac{1}{p_{a^*}},
\end{equation}
where again the last equality assumes the optimal choice of $C$. 
Even though the cost involved in proposing the $B$ is in general \emph{not negligible}, in Section \ref{sect b=1} we will give several natural examples, similar to Algorithm~\ref{IP PDC DSH procedure}, where it is.

\begin{remark}
For each proposed value $a=A_i$, we need to be able to compute $t(a)$; this can be either a minor cost, a major cost, or an absolute impediment, making probabilistic divide-and-conquer infeasible. All of these variations occur in the context of integer partitions, and will be discussed in Sections
\ref{sect divide} -- \ref{sect d(z)}, and again in Section \ref{sect cost}; see also Remark~\ref{X}.
\end{remark}

\begin{remark}
\noindent {\bf Hard versus soft rejection.} Adopting the language of information theory, where soft-decision decoders are contrasted with hard-decision decoders, by a \emph{hard rejection} function we mean an acceptance/rejection function $t : \A \to [0,1]$ whose image is actually $\{0,1\}$.  We may thus consider Algorithm~\ref{IP WTGL procedure} to be more specifically a \emph{hard} rejection sampling algorithm.  This is in contrast to Algorithm~\ref{IP PDC DSH procedure}, which may be considered a \emph{soft} rejection sampling algorithm.  

\end{remark}

\section{Algorithms for simulating random integer partitions of $n$}
\label{Algorithms Section}

 
The computational analysis that follows uses an informal adaptation of \emph{uniform} costing; see Section \ref{sect cost}.  Some elements of the analysis, specifically asymptotics for the acceptance rate, 
will be given rigorously, for example in Theorems \ref{lucky theorem}, \ref{b=1 theorem}, and \ref{recursive theorem}.

An integer partition, of size $n$, is a multiset of positive integers, with  sum $n$.  We denote the number of partitions of a given size~$n$ by $p(n)$. 
We typically denote the parts of an integer partition by $\lambda_1 \geq \lambda_2 \geq \ldots \geq \lambda_\ell>0$, where $\ell$ is the number of parts in the partition, which varies from $1$ to $n$.  Instead of listing the part sizes, we can instead record the multiplicities $c_1, \ldots, c_n$ of each number $1, \ldots, n$ in the integer partition.  

For example, the integer partition of size~12 with parts $(5, 3, 2, 1, 1)$ can be equally described by the $12$--tuple $(c_1,\ldots,c_{12})  = (2, 1, 1, 0, 1,0,\ldots,0)$, where $c_i$ denotes the number of parts of size $i$.  
All partitions of a given size $n$ satisfy $\sum_{i=1}^n  i\, c_i = n$, and this motivates the probabilistic formulation, given in Section~\ref{sect grand canonical}, via formulas~\eqref{def T}~and~\eqref{prob lambda}. 

\subsection{For baseline comparison: table methods}\label{sect table} 
A natural algorithm for the generation of a random integer partition 
is to find the largest part first, then the second largest part, and so on.
\ignore{ according to a distribution of largest part
\footnote{For the largest part, we are are dealing with \emph{unrestricted} partitions of $n$, but for subsequent parts, the problem involves the largest part of a partition of an integer $m$ into parts of size at most 
 $k$.}.} The main cost associated with this method is the storage of all the distributions.
 Some details follow.
 
Let $p(\leq k, n)$ denote the number of partitions of $n$ with each part of size $\leq k$, so that $p(\le n,n) = p(n)$.  These can
be quickly calculated from the recurrence
\begin{equation}
\label{k parts recurrence}
p(\leq k, n) = p(\leq k-1,n) + p(\leq k,n-k),
\end{equation}
where the right hand side counts the number of partitions without any $k$'s plus the number of partitions with at least one $k$.  
\ignore{
Let $X_i$ denote the $i$-th largest part of a randomly generated partition, so $\lambda = (X_1,X_2,\ldots)$.  We have
$$
\p(X_1 \leq i) = p(\le i,n)/p(n),
$$ 
$$
\p(X_2 \leq i|X_1 = k) = p(\leq i,n-k)/p(\leq k, n-k),
$$
$$
\p(X_3 \leq i|X_1 = \ell,X_2=k) = p(\leq i,n-\ell-k)/p(\leq k, n-\ell-k),
$$
and so on.
 The first line on display above deals with \emph{unrestricted} partitions of $n$; subsequent lines involve partitions of an integer $j$ into parts of size at most 
 $k$.
 }
 If only a few partitions are desired, the quantities on the right--hand side above can be computed as needed via the recursion given in Equation \eqref{k parts recurrence} and normalized to form a probability distribution.  However,
 rather than computing each quantity as it appears, an $n$-by-$n$ table, whose $(i,j)$ entry is  $p(\leq i, j)$, can be computed and stored. The generation of random partitions is extremely fast once this table has been created.  
\ignore{There is a variation to economize on storage\footnote{Since the largest part of a random partition is extremely unlikely to exceed
 $O(\sqrt{n}\log n)$ (by \cite{fristedt}), one can get away using a table of size $O(n^{3/2}\log n)$, which will only rarely need to be augmented.
 Specifically, writing $\lambda_1'$ for the largest part of a partition, with $c=\pi/\sqrt{6}$, for fixed $A>0$, with $i_0 = i_0(n,A) := \sqrt{n} \log(A \sqrt{n}/c)/c$, 
$\p_n(\lambda_1' \ge i_0) \to 1-\exp(-1/A)$.  For example, take $A=1000$, so that the $i_0$ by $n$ table will need to be augmented with probability approximately $.001$.  With $M$ words of memory available for the table, we solve $i_0(n,A) \times n = M$;  for example, with $M=2^{28}$ and $A=1000$ we have $n=15000 \doteq 2^{17}$ and $i_0=1700$, and increasing $M$ to $2^{37}$ gets us up to $n=9\times 10^6 \doteq 2^{23}$, $i_0=15000$.  Instead of actually augmenting the table, one could treat the roughly one out of every $A$
computationally difficult cases as \emph{deliberately missing data}, as in Section \ref{sect missing}.
}
and another variation to improve the number of lookups.\footnote{Using \eqref{k parts recurrence} directly, there is one lookup for each part of the random partition, by \cite{lehner,fristedt},
 this is order of $\sqrt{n} \, \log n$ lookups.  An easy variation simultaneously finds the multiplicity of the largest part, improving the order of lookups to order of  $\sqrt{n}$.
}
 }

An algorithm for simulating a random partition, based on Euler's identity $n p(n) = \sum_{d,j \ge 1} d p(n-d\,j)$, is given in
\cite{NWarticle,NWbook}, and cited as ``the recursive method.''   
We haven't found a clearcut complexity analysis in the literature, although \cite{Zimmermann} comes close.  But we believe that this algorithm is less useful than the $p(\le k,n)$ table method --- sampling from the distribution on $(d,j)$ implicit in $n p(n) = \sum_{d,j \ge 1} d\, p(n-d\, j)$ requires computation of partial sums;  if all the partial sums for $ \sum_{d,j \ge 1, dj \le m} d\, p(m-d\, j)$ with $m \le n$ are stored in a table, the total storage requirement is of order $n^2 \log n$, and if they aren't stored, computing the values, as needed, becomes a bottleneck.

\subsection{Rejection sampling}\label{sect grand canonical}

Fristedt \cite{fristedt} observed that, for any choice $x \in (0,1)$, if $Z_i \equiv Z_i(x)$ has the geometric distribution given by
\begin{equation}\label{dist Z}
 \p(Z_i=k) \equiv  \p_x(Z_i=k)  = (1-x^i) \ (x^i)^k, \ k=0,1,2,\ldots,
\end{equation}
with $Z_1,Z_2,\ldots$ independent, and $T$ is defined by  
\begin{equation}\label{def T}
    T= Z_1 + 2 Z_2 + \cdots,
\end{equation}
then, conditional on the event $\{T=n\}$, the partition $\lambda$ having $Z_i$ parts of size $i$, for $i=1,2,\ldots$,
is uniformly distributed over the $p(n)$ possible partitions of $n$.\footnote{We write $\p_x$ or $\p$, and $Z_i$ or $Z_i(x)$ interchangeably,
depending on whether the choice of $x \in (0,1)$ needs to be emphasized, or left understood.}
 
This extremely useful observation is easily seen to be true, since for any nonnegative integers $(c_1,c_2,\ldots)$ with $c_2 + 2 c_2 + \cdots=n$, specifying a partition $\lambda$ of the integer $n$,
$$
\p(Z_i=c_i, i=1,2,\ldots) = \prod \p(Z_i = c_i) = \prod \left( (1-x^i) (x^i)^{c_i} \right) 
$$
\begin{equation}\label{prob lambda}
=  \ x^{c_1 + 2 c_2 + \cdots} \prod(1-x^i)= x^n \  \prod(1-x^i),
\end{equation}
which does not vary with the partition $\lambda$ of $n$.

The event $\{T=n\}$ is the disjoint union, over all partitions $\lambda$ of $n$, of the events whose probabilities are given
in \eqref{prob lambda}, showing that
\begin{equation}\label{prob T=n}
  \p_x(T=n) = p(n)  \ x^n \ \prod(1-x^i)  .
\end{equation}

If we are interested in random partitions of $n$, an especially effective choice for $x$, used in \cite{fristedt,shape}, is
\begin{equation}\label{def x(n)}
   x(n) = \exp( -c /\sqrt{n}),     \ \mbox{ where } c = \pi/\sqrt{6}.
\end{equation}
Under this choice, we have, as $n \to \infty$,
\begin{equation}\label{mean and var}
\frac{1}{n} \ \e_{x(n)} T \to 1, \ \ \frac{1}{n^{3/2}} \var_{x(n)} T \to \frac{2}{c};
\end{equation}
this is essentially a pair of Riemann sums, see \cite[page 106]{IPARCS}.
If we write $\sigma(x)$ for the standard deviation of $T$, then the second part of \eqref{mean and var}
says:
with $x=x(n),$ as $n \to \infty$,
\begin{equation}\label{SD}
 \sigma(x) \sim \sqrt{2/c} \ n^{3/4}.
\end{equation}
The \emph{local central limit heuristic} would thus suggest  asymptotics for $\p_x(T=n)$, and these simplify, using
\eqref{def x(n)} and \eqref{mean and var}, as follows:
\begin{equation}\label{prob T=n asym}
 \p_x(T=n) \sim \frac{1}{ \sqrt{2 \pi} \ \sigma(x)} \sim \frac{1}{2\sqrt[4]{6}\, n^{3/4}}.
 \end{equation}
 
Hardy and Ramanujan \cite{HR18} proved the asymptotic formula, as $n \to \infty$,
\begin{equation}\label{HR}
p(n) \sim \exp(2 c \sqrt{n})/ (4 \sqrt{3}n), \ \mbox{ where } c = \pi/\sqrt{6} \doteq 1.282550.
\end{equation}
The Hardy--Ramanujan asymptotics \eqref{HR} and the exact formula \eqref{prob T=n} combine to show that \eqref{prob T=n asym} does hold.

\ignore{
\begin{theorem}\label{lucky theorem}{\bf Analysis of \emph{Waiting-to-get-lucky}.}

Consider the following algorithm to generate a random partition of $n$, chosen uniformly from the $p(n) \sim 
\exp(2 c \sqrt{n})/(4 \sqrt{3}\, n)$ possibilities.  Use the distributions in \eqref{dist Z}, with parameter $x$ given by
\eqref{def x(n)}.

\begin{enumerate}
\item Propose a sample, $Z_1,Z_2,\ldots,Z_n$;  compute $T_n := Z_1 + 2 Z_2 + \cdots + n Z_n$.

\item In case $(T_n=n)$, we have \emph{got lucky}.  Report the partition $\lambda$ with $Z_i$ parts of size $i$, for $i =1$ to $n$, and stop.  Otherwise, repeat from the beginning.

\end{enumerate}
}
\begin{theorem}\label{lucky theorem}{\bf Analysis of Rejection Sampling.}

Algorithm~\ref{IP WTGL procedure} \emph{does} produce one sample from the desired distribution, and  the expected number of proposals until the algorithm terminates is asymptotically $2\sqrt[4]{6}\, n^{3/4}$.

\end{theorem}
\begin{proof}
It is easily seen that $\p_x(Z_i=0$ for all $i > n) \to 1$.
Hence, the asymptotics \eqref{prob T=n asym},
given for the infinite sum $T$, also serve for the finite sum $T_n$, in which the number of summands, along with the parameter $x=x(n)$, varies with $n$.
\ignore{It is easily seen that $\p_x(Z_i=0$ for all $i > n) \to 1$. [In detail, $\p($not $ (Z_i=0$ for all $i > n)) \le \sum_{i>n} \p(Z_i \ne 0) = \sum_{i>n} x^i < \sum_{i \ge n} x^i = x^n/(1-x) \sim x^n / (c/\sqrt{n}) =  \sqrt{n}\exp(-c \sqrt{n})/c \to 0$.]    Hence the asymptotics \eqref{prob T=n asym},
given for the infinite sum $T$, also serve for the finite sum $T_n$, in which the number of summands, along with the parameter $x=x(n)$, varies with $n$.
} 
\end{proof}

\begin{remark}\label{floating point remark}
We are \emph{not} claiming that the running time of the algorithm grows like $n^{3/4}$, but only that the number of proposals needed to get one acceptable sample grows 
like $n^{3/4}$.  
The time to propose a sample also grows with $n$.    Assigning cost 1 to each call to the random number generator, with all other operations being free,
 the cost to propose one sample grows like $\sqrt{n}$ rather than $n$; details in Section \ref{sect cost}.
  Combining with Theorem \ref{lucky theorem}, the cost of 
 the rejection sampling 
algorithm 
grows like $n^{5/4}$.  

\end{remark}

\subsection{Divide-and-conquer with a deterministic second half} \label{sect b=1}

Perhaps surprisingly,
the choice 
$A = (Z_2, \ldots, Z_n)$ and $B = Z_1$ from Algorithm~\ref{IP PDC DSH procedure} 
is excellent.  Loosely speaking, it reduces the cost of rejection sampling from  order $n^{3/4}$ to order $n^{1/4}$.

\ignore{
Recall, stage 1, for the algorithm of Section \ref{subsection basic},  was to propose (one or $m$ of) $A=(Z_{b+1},Z_{b+2},\ldots,Z_n)$, and convert the distribution of $A$ to that of $X$, using using acceptance/rejection sampling, 
then in stage 2, having observed $A=a$ to find  matching $B=(Z_1,\ldots,Z_b)$.  Here, with $b=1$, the distribution of $Y|X=a$, i.e., the distribution of $B$  conditional on $h(a,B)=1$, is trivial. All the work is in stage 1.   
} 

The analysis of the speedup relative to waiting-to-get-lucky, as defined in \eqref{speedup},  is easy.  
\begin{theorem}\label{b=1 theorem}
The speedup, as defined in \eqref{speedup}, of Algorithm~\ref{IP PDC DSH procedure}, 
relative to the rejection sampling algorithm of Algorithm~\ref{IP WTGL procedure},
is asymptotically $\sqrt{n}/c$, with $c=\pi/\sqrt{6}$.  Equivalently, the acceptance cost is
asymptotically $2 \, \pi\, \sqrt[4]{6}\, n^{1/4}$.
\end{theorem}
\ignore{\begin{proof}

We have $A =  (Z_2, \ldots, Z_n)$, $B = Z_1$, whence,
\[ \e h(a^\ast, B) = \max_\ell \p(Z_1 = \ell) = 1-x.  \]
Recall that  $x=e^{-c/\sqrt{n}}$, where $c = \pi/\sqrt{6}$.  Let $a = (z_2,\ldots,z_n)$.  We have
\begin{eqnarray*}\displaystyle 
t(a) = &
 \frac{\p(Z_1 = n-\sum_{i=2}^n i\, z_i)}{\max_\ell \p(Z_1 = \ell)} = e^{-k\, c/\sqrt{n}}, \\
\text{speedup} = & \frac{1}{\max_\ell \p(Z_1 = \ell)} = (1-x)^{-1} \sim \frac{\sqrt{n}}{c},
\end{eqnarray*}
where $k=n-\sum_{i=2}^n i\, z_i.$

\ignore{
Recall that  $x=e^{-c/\sqrt{n}}$, where $c = \pi/\sqrt{6}$. 
Let $k = n - \sum_{i=2}^n i\, z_i$, then we have
\[ t(a) = e^{-k\, c/\sqrt{n}} \]
\[ speedup = (1-x)^{-1} \sim \frac{\sqrt{n}}{c} \]
}
\end{proof}
}
\begin{proof}
We have $A = (Z_2, \ldots, Z_n)$, $B = Z_1$,
so with $a=(z_2, \ldots, z_n)$  and  $\ell = n-(2z_2 + \cdots+nz_n),$ we have
$h(a,B) = \p(Z_1=\ell)$.   Following \eqref{speedup}, and recalling that $x=e^{-c/\sqrt{n}}$, where $c = \pi/\sqrt{6}$, we have
$$
\text{speedup} =\frac{1}{\max_\ell \p(Z_1 = \ell)} = \frac{1}{ \p(Z_1 = 0)} = \frac{1}{1-x} \sim \frac{\sqrt{n}}{c}.
$$
Combined with the asymptotic rejection cost for  Algorithm~\ref{IP WTGL procedure},
as given by Theorem \ref{lucky theorem}, this yields the claimed asymptotic acceptance
cost for  Algorithm~\ref{IP PDC DSH procedure}.
\end{proof}



 To review, step 1 of  Algorithm~\ref{IP PDC DSH procedure}
 is to simulate $(Z_2,Z_3,\ldots,Z_n)$, and accept it with probability \emph{proportional to} the chance that
$Z_1 = n-(2Z_2+\cdots + n Z_n)$.   
The speedup shows the brilliance of the idea in \cite{Rejection}: conditional on accepting a value $(z_2, \ldots z_n)$, we are done;  the distribution $\L(Z_1|Z_1 + \sum_{i=2}^n i\, z_i = n)$ is trivial,
so there is no need to ``wait for a lucky $Z_1$".
 Algorithm~\ref{IP PDC DSH procedure} is a soft rejection sampling algorithm;  it might be more appropriate to call it a
\emph{soft acceptance} sampling algorithm.
In contrast, the hard rejection in Algorithm  ~\ref{IP WTGL procedure}
can be viewed as  simulating  $(Z_2,Z_3,\ldots,Z_n)$ and then \emph{simulating} $Z_1$ to see whether or not $Z_1 = n-(2Z_2+\cdots + n Z_n)$;  if $Z_1$ does not have the correct value, then one resamples $(Z_2,Z_3,\ldots,Z_n)$, (and then  $Z_1$), until finally getting lucky.

\subsubsection{Examples: $k$-cores, set partitions, and plane partitions}\label{additional examples}
The PDC using a deterministic second half is remarkably robust and powerful.  It is robust in that it applies, and gives a nontrivial speedup, for a wide variety of simulation tasks.  

For a first additional example, there is much recent interest in $k$-cores, see \cite{kcore}.  There is an easy-to-calculate bijection between $k$-cores and integer partitions with $\lambda_1<k$.    The natural way to simulate partitions of $n$ with $\lambda_1 < k$, following Fristedt's method from Section \ref{sect grand canonical}, is to use
independent $Z_1,\ldots,Z_{k-1}$, but, instead of defining $x=x(n)$ as in \eqref{def x(n)}, we instead 
solve numerically for $x=x(k,n)$ which satisfies $n = \sum_{1 \le i < k} \e i Z_i =
\sum_{1 \le i < k} i\, x^i/(1-x^i)$.   Using $A=(Z_2,\ldots,Z_{k-1})$ and $B=Z_1$, the optimal threshold function is explicitly given by $\p(Z_1 = m) / \max_j\p(Z_1 = j)$.  
The speedup, relative to 
rejection sampling, 
is
$1/(1-x(k,n))$ --- exactly the same form as in Theorem \ref{b=1 theorem}, except that we no longer have the asymptotic analysis that
$1/(1-x) \sim \sqrt{n}/c$.  

For a second example, which shows that a good deterministic second half is not always the first component, consider the integer partition underlying a random set partition, where
all 
set partitions 
are equally likely.  See \cite[Section 10.1]{IPARCS}, for more details;  the summary is that with $x=x(n)$ to solve $xe^x=n$, so that $x \sim \log n$, we want independent $Z_i$ where $Z_i$ is Poisson distributed with parameter $\lambda_i=x^i/i!$, and getting lucky is to have $\sum_{i=1}^n i Z_i = n$.  For the deterministic second half, we have $B=Z_{j}$, where we pick $j$ to maximize the speedup. 
Thus we want to minimize
$\max_k \p(Z_j=k)$, i.e., to maximize $\lambda_j$, so we take $j=\lfloor x \rfloor$.  Since $j = x + o(\sqrt{x})$, we have $\lambda_{j}=x^j/j! \sim e^x/\sqrt{2 \pi x}$, and since $e^x=n/x$, we have $\lambda_j \sim n/\sqrt{2 \pi x^3}$.  
As in the previous paragraph, again the optimal threshold function is given by an explicit expression which is easy to evaluate. 
The speedup factor is $1/\max_k \p(Z_{j}= k)$, which is asymptotic to $\sqrt{2 \pi \lambda_j} \sim (2\pi)^{1/4}\sqrt{n} / x^{3/4}$.  Pittel showed \cite{pittelset} that the expected number of proposals for 
rejection sampling 
is asymptotic to $ \sqrt{2 \pi n (x+1)} $.  Hence the expected number of proposals using the deterministic second half PDC is asymptotically $ x^{5/4} / (2\pi)^{1/4}$, which is $O( \log^{5/4}(n) )$.  

\ignore{
This example provides an 
 illustration of \emph{soft} versus \emph{hard} rejection, from the end of Section
\ref{sect acceptance}.  
Recall that $T_n:=\sum_{1}^n i Z_i$ and we are trying to get an instance of $(Z_1,Z_2,\ldots,Z_n)$ for which $T_n=n$.  
When a value $a$ for $A=  (Z_1,Z_2,\ldots,Z_{j-1},Z_{j+1},Z_{j+2},\ldots,Z_n)$ has been proposed, we know the 
weighted sum of $A$, 
say 
\begin{equation}\label{set color}
 T_A = \sum_{i: i \ne j} i Z_i =  T_n - j Z_j, 
\end{equation}
and
$B$ satisfies $h(a,B)=1$ if and only if $j Z_j = n-T_A$.   Much of the time, we can hard-reject $A=a$ since either $T_A>n$  or $(n-T_A)/j$ is not an integer.  Computing roughly, $\p(T_A \le n, (n-T_A)/j \in \BZ) \doteq \p(T_A \le n) \times \p ((n-T_A)/j \in \BZ) \doteq
\frac{1}2 \times \frac{1}j \sim 1/(2x)$.
The overall acceptance cost  is at least as great as the acceptance cost of evading the hard rejection;  here we have made the plausiblity check on the asymptotics,  with $ x^{5/4} / (2\pi)^{1/4} \ge 2x$.  There is an interplay between these hard rejection thresholds and total variation distance; see \cite[ equations (35) and (36)]{IPARCS}.
}

For a third example, we consider a sampling procedure in \cite{bodini}, in which a random plane partition is generated in two stages.
Stage 1 is rejection sampling, 
 where the proposal is an array of  independent geometrically distributed random variables $\{Z_{i,j}\}$, $0 \leq i, j < n$.\footnote{With $\p(Z_{i,j}\ge k)=(x^{i+j+1})^k$ and $x=x(n) = \exp(-(2 \zeta(3)/n)^{1/3})$, akin to \eqref{dist Z} and \eqref{def x(n)}.}   To get an instance of weight $n$, the expected number of proposals, using
 rejection sampling, 
  is of order $n^{2/3}$. There is an algorithm, to make the proposal, whose work is of the same order as the entropy lower bound, which is order of $n^{2/3}$, so the cost for stage 1 is order of $n^{4/3}$.
Stage 2 is to apply a bijection due to Pak \cite{Pak}, which, as implemented in \cite{bodini} takes order of $n \log^3(n)$.  
Combined, \cite{bodini} has order $n^{4/3}$ to name a random ensemble of weight $n$, plus order $n \log^3n$ to implement the bijection, so the order $n^{4/3}$ first stage dominates the computation.  Using our deterministic second half  PDC, with $B= Z_{0,0}$, we obtain a speedup of order  $n^{1/3}$, bringing the cost of stage 1 
to $O(n)$, and hence bringing the total cost of the algorithm down to the cost of implementing the bijection, viz., order of $n \log^3n$.  

 \subsection{Divide-and-conquer,  by small versus large} \label{sect divide}
\ignore{
Rejection sampling 
is limited primarily by the probability that the target is hit, which diminishes like $n^{-3/4}$.  Already at $n=10^8$, the probability is one in a million.  
Instead of trying to hit a hole in one,  we allow approach shots.   

Recall that to sample partitions uniformly, based on  \eqref{dist Z} -- \eqref{prob T=n asym},      our goal is to sample from the distribution of $(Z_1,Z_2,\ldots,Z_n)$ conditional on $\{T=n\}$.  
}
We now return to the example of random sampling of integer partitions. 
Using  $x=x(n)$ from  \eqref{def x(n)}, for any fixed choice
 $b \in \{1,2,\ldots,n-1\}$, we let 
\begin{equation}\label{b}
A = (Z_{b+1},Z_{b+2},\ldots,Z_{n}), \ \ \ B = (Z_1,\ldots,Z_b),
\end{equation}
noting that one extreme case, $b=1$, was handled in Section~\ref{sect b=1}.
   With 
\begin{equation}\label{def TAB}
T_A = \sum_{i=b+1}^n i Z_i, \ \ \ \ T_B = \sum_{i=1}^b i Z_i, 
\end{equation}
we want to sample from
$(A,B)$ conditional on $(T_A+T_B=n)$.  The standard paradigm for deterministic divide-and-conquer, that the two tasks should be roughly equal, 
suggests $b$ of order   $\sqrt{n}$.  

\ignore{
We have $A,B$ independent, and we use $h(A,B)=1(T_A+T_B=n)$.
The divide-and-conquer strategy, according to Lemma \ref{lemma 1},  is to sample  $X$ from the  distribution of $A$ biased by the probability that an independently chosen $B$ will make a match, and then, having observed $(A=a)$, sample $Y$ from the distribution of  $B$ conditional on $(h(a,B)=1)$.
}

 In order to simulate $X$, we will use rejection sampling, as reviewed in Section \ref{sect acceptance}.  
 To find the optimal rejection probabilities, we want the largest $C$ such that 
$$
C \max_{j}\frac{\p(T_B = n-j)}{\p(T_n=n)} \leq 1,
$$
or equivalently,
\begin{equation}\label{C}
C = \frac{\p(T_n=n)}{\max_j \p(T_B=j)}.
\end{equation}

Once an $A$ has been accepted, and we have our target $n-T_A$ for the sum $T_B$, we simply propose 
$B = (Z_1, Z_2, \ldots, Z_b)$ until $Z_1+2Z_2+\cdots+bZ_b=n-T_A$; then the pair $(A,B)$ is our random partition.


\begin{remark}
\label{b remark}
The values $\p(T_B=j)$ for $j=0,1,\ldots,n$ can be  computed using the recursion 
\eqref{k parts recurrence};   what we really have is a variant of the $n$-by-$n$ table method of Section
\ref{sect table}, in which the table is $b$ by $n$. The computation time for the entire table is $b \, n$.  However, one only needs to store the current and previous row, (or with overwriting, only the current row), so the storage is $n$.  Once we have the last row of the $b$ by $n$ table, we can easily find $C$ and indeed the entire threshold function $t$.  
\end{remark}

\subsection{Self-similar, recursive probabilistic divide-and-conquer: $p(z) = d(z)p(z^2)$} \label{sect d(z)}
The methods of Sections \ref{sect grand canonical} -- \ref{sect b=1} have acceptance costs that go to infinity with $n$.  We now demonstrate a recursive probabilistic divide-and-conquer algorithm that has an asymptotically constant acceptance cost.

A well-known result in partition theory is
\begin{equation}
\label{pz squared}
p(z) = \prod_i (1-z^i)^{-1} =\left(\prod_i 1+z^i\right)  \left(\prod_i \frac{1}{1-z^{2i}}\right) =d(z) p(z^2),
\end{equation}
where $d(z) = \prod_i 1+z^i$ is the generating function for the number of partitions with distinct parts, and $p(z^2)$ is the generating function for the number of partitions where each part has an even multiplicity.  This can of course be iterated to, for example, $p(z) = d(z)d(z^2)p(z^4)$, etc., and this forms the basis for a recursive algorithm.

 Recall from \eqref{dist Z} in Section \ref{sect grand canonical}, that  each $Z_i \equiv Z_i(x)$ is geometrically distributed with 
$\p(Z_i \ge k) =x^{ik}$.  The \emph{parity bit} of $Z_i$, defined by
$$ \epsilon_i = 1(Z_i \mbox{ is odd}),
$$
is a Bernoulli random variable $\epsilon_i \equiv \epsilon_i(x)$, with 
\begin{equation}\label{def epsilon}
\p(\epsilon_i(x)=1) = \frac{x^i}{1+x^i}, \ \ \p(\epsilon_i(x)=0) = \frac{1}{1+x^i}.
\end{equation}
Furthermore, $(Z_i(x) -\epsilon_i)/2$ is geometrically distributed, as $Z_i(x^2)$, again in the notation \eqref{dist Z}, and
$(Z_i(x) -\epsilon_i)/2$ is independent of $\epsilon_i$.  What we really use is the converse:  with $\epsilon_i(x)$ as above, independent of $Z_i(x^2)$, the $Z_i(x)$ constructed as
$$
Z_i(x) := \epsilon_i(x) + 2 Z_i(x^2), \ \ \ i=1,2,\ldots
$$
indeed has the desired geometric distribution.
\ignore{ FOR NOW
The weighted sum of the $A$ phase, namely, $T_A$, has expectation 
$$
\e T_A = \sum_{i=1}^n \frac{ix^i}{1+x^i},
$$
 which, according to \cite[page 107]{IPARCS}, was shown to be asymptotic to $n/2$ for the usual choice of $x=\exp(-c/\sqrt{n})$, with standard deviation $\sqrt{2}\, \sigma(x)$ (see Equation \eqref{SD}).
}

\begin{theorem}\label{recursive theorem}

The asymptotic acceptance cost for one step of the recursive divide-and-conquer algorithm using $A=(\epsilon_1(x),\epsilon_2(x),\ldots)$, $B = ((Z_1(x)-\epsilon_1)/2,(Z_2(x)-\epsilon_2)/2,\ldots)$, is $\sqrt{8}$.
\end{theorem}

\begin{proof}
The acceptance cost $1/C$ can be computed via \eqref{C} and \eqref{prob T=n asym}, with 
\begin{eqnarray}
\nonumber C &=&   \frac{ \p_x(T=n)}{\max_k \p_{x^2}(T = \frac{n-k}{2})} = 
\frac{\p_x(T=n)}{\max_k \p_{x^2}(T = k)} \sim \frac{\frac{1}{\sqrt{2\pi}\, \sigma(x)}}{\frac{1}{\sqrt{2\pi}\, \sigma(x^2)}}\\
\nonumber  & \sim & \frac{ n^{-3/4}}{(n/4)^{-3/4}} = 4^{-3/4} = 8^{-1/2}.
\end{eqnarray}
\end{proof}

In effect, the recursive algorithm is to determine the $(Z_1,Z_2,\ldots,Z_n)$ conditional on $(Z_1+2Z_2+\ldots=n)$, by finding the binary expansions:  first the 1s bits of all the $Z_i$s, then the 2s bits, then the 4s bits, and so on. 
Note that the rejection probabilities can be computed in time which is little-o of the cost to generate the random bits; see Remark~\ref{X} and Section~\ref{sect cost}. 

\ignore{ 
With a little more detail:  to start, with $A = (\epsilon_1(x),\epsilon_2(x),\ldots)$ and $T_A = \sum_i i \, \epsilon_i$, we have 
$\e T_A = \sum_{i=1}^n \frac{ix^i}{1+x^i} \sim n/2$, and it can be shown that, even after conditioning on acceptance, the distribution of $T_A$
is concentrated around $n/2$. Since $B = ((Z_1(x)-\epsilon_1)/2,(Z_2(x)-\epsilon_2)/2,\ldots)$ is equal in distribution to  $(Z_1(x^2), Z_2(x^2), \ldots)$, and target
$n'=(n-T_A)/2$, we see that the $B$ phase is to find a partition of the integer $n'$, uniform over the $p(n')$ possibilities.   In carrying out the B task we simply use
$x(n')$ as the parameter, but the choice
$(x(n))^2$ would also work.    
} 

\subsubsection{Exploiting a parity constraint} 
\label{sect parity}

Theorem \ref{recursive theorem} states that the asymptotic acceptance cost for  proposals  of $A = (\epsilon_1(x),\epsilon_2(x),\ldots)$ is $2 \sqrt{2}$, and this already takes into account an obvious lower bound of 2, since the parity of
$T_A =\epsilon_1+\epsilon_2+\cdots + \epsilon_n$ is nearly  equally distributed over \{odd, even\}, and rejection is guaranteed 
if $T_A$ does not have the same parity as $n$.  An additional speedup is attainable by moving $\epsilon_1$ from the $A$ side to the $B$ side:  instead of simulating $\epsilon_1$, there now will be a trivial task, just as there was for $Z_1$ in the ``$b=1$" procedure of Section \ref{sect b=1}.  That is,  we switch to $A=(  \epsilon_2(x),\epsilon_3(x),\ldots)$ 
and $B = (\epsilon_1(x),(Z_1(x)-\epsilon_1)/2,(Z_2(x)-\epsilon_2)/2,\ldots)$; the  parity of the new $T_A$ dictates, deterministically, the value of the first component of $B$, under the conditioning on $h(a,B)=1$.  The
rejection probabilities for a proposed $A$ are like those in Theorem \ref{recursive theorem}, but with an additional factor of $1/(1+x)$ or
$x/(1+x)$,  depending on the parity of $n+\epsilon_2+\cdots + \epsilon_n$.  Since $x=x(n) \to 1$ as $n \to \infty$, these two factors both tend to 1/2, so the constant $C$ as determined by \eqref{C} becomes, asymptotically, twice as large.
 
 \begin{theorem}\label{prop root 2}
The asymptotic acceptance cost for one step of the recursive divide-and-conquer algorithm using $A=(\epsilon_2(x),\epsilon_3(x),\ldots)$  and   $B = (\epsilon_1(x),(Z_1(x)-\epsilon_1)/2,(Z_2(x)-\epsilon_2)/2,\ldots)$ is $\sqrt{2}$.
\end{theorem}

\begin{proof}
The acceptance cost $1/C$ can be computed, as in the proof of Theorem \ref{recursive theorem}, with the only change being that in the display for computing $C$, the expression under the  $\max_k$, which was $\p(T(x^2) = (n-k)/2)$ changes to  
$$ \p\left(2 | n-k +\epsilon_1(x) \mbox{ and }T(x^2) = \left\lfloor \frac{n-k}{2}\right\rfloor\right) $$
$$ =\p(  2 | n-k +\epsilon_1(x)   ) \times \p\left(T(x^2) =   \left\lfloor\frac{n-k}{2}\right\rfloor\right)  \sim \frac{1}{2}  \p\left(T(x^2) =   \left\lfloor\frac{n-k}{2}\right\rfloor \right) .
$$
 \end{proof}


\subsubsection{The overall cost of the main problem and all its subproblems}

Informally, for the algorithm in the previous section, the main problem has size $n$ and acceptance cost $\sqrt{2}$, applied to a proposal cost asymptotic to $c_0 \sqrt{n}$, for a net cost $\sqrt{2}\, c_0 \sqrt{n}$.  The first subproblem has random size, concentrated around $n/4$, and hence half the cost of the main problem.  The sum of a geometric series with ratio $1/2$ is twice the first term, so the net cost of the main problem and all subproblems combined is $2 \sqrt{2}\, c_0 \sqrt{n}$.

In framing a theorem to describe this, we try to allow for a variety of costing schemes.  We believe that the first sentence in the hypotheses of Theorem
\ref{theorem sum of series} is valid, with $\theta=1/2$, for the scheme of
Remark \ref{floating point remark}.  The second sentence, about costs of tasks other than proposals, is trivially true for the scheme of Remark \ref{floating point remark}, but may indeed be false in
costing schemes which assign a cost to memory allocation, and communication.

\begin{theorem}\label{theorem sum of series}
\emph{Assume} that the cost $C(n)$ to propose the
$A = (\epsilon_2(x),\epsilon_3(x),\ldots)$ in the first step of the
algorithm of Section \ref{sect parity} is given by a deterministic function with $C(n)  \sim c_0  n^\theta$ for some positive constant $c_0$ and constant $\theta \ge 1/2$, or even more generally, $C(n) =  n^\theta$ times a slowly varying function of $n$.   Assume that the cost of all steps of the algorithm, other than making proposals,\footnote{such as the arithmetic to compute acceptance/rejection thresholds, the generation of the random numbers used in those acceptance/rejection steps, and merging the accepted proposals into a single partition.} is relatively negligible, i.e., $o(C(n))$.  
Then, the asymptotic cost of the 
entire algorithm is
$$\frac{1}{1-1/4^\theta} \, \sqrt{2}  C(n) \le 2 \,  \sqrt{2}\, C(n) .
$$
\end{theorem}

\begin{proof}
The key place to be careful is in the distinction between the distribution of a proposed $A= (\epsilon_2(x),\epsilon_3(x),\ldots)$, and the distribution after rejection/acceptance.  For proposals, in which the $\epsilon_i$ are independent, with $T_A :=\sum_2^n i \, \epsilon_i(x)$ and $x=x(n)$ from \eqref{def x(n)}, calculation gives $\e T_A \sim n/2$ and $\var(T_A) \sim (1/c) n^{3/2}$.  Chebyshev's inequality for being at least $k$ standard deviations away from the mean, to be used with $k=k(n) = o(n^{1/4})$ and $k \to \infty$, gives $\p( |T_A - \e T_A | > k \, \SD(T_A) ) \le 1/k^2$.

Now consider the good event $G$ that a proposed $A$ is accepted;  conditional on $G$, the $\epsilon_i$ are no longer independent.  But the upper bound from Chebyshev is robust, with $\p( |T_A - \e T_A | > k \, \SD(T_A) | G ) \le 1/(k^2 \, \p(G))$.  Since $\p(G)$ is bounded away from zero, by Theorem \ref{prop root 2}, we still have an upper bound which tends to zero, and shows that $(n-T_A)/2$, divided by $n$, converges in probability to $1/4$. 

Write $N_i \equiv N_i(n)$ for the random size of the subproblem at stage $i$, starting from $N_0(n)=n$.  The previous paragraph showed that for $i=0$, $N_{i+1}(n)/N_i(n) \to 1/4$, where the convergence is \emph{convergence in probability}, and the result extends automatically to each fixed $i=0,1,2,\ldots$. 
We have deterministically that $N_{i+1}/N_i \le 1/2$, so in particular $N_i>0$ implies $N_{i+1} < N_i$.  Set $C(0)=0$, redefining this value if needed, so that the costs of all proposals is exactly the random 
$$
 S(n) :=  \sum_{i \ge 0} C(N_i(n)).
$$
It is then routine analysis to  use the hypothesis that $C(n)$ is regularly varying to conclude that $S(n)/C(n) \to 1/(1-4^{-\theta})$,
where again, the convergence is convergence in probability.   The deterministic bound $N_{i+1}(n)/N_i(n) \le 1/2$ implies
that the random variables $S(n)/C(n)$ are \emph{bounded}, so it also follows that $ \e S(n)/C(n) \to 1/(1-4^{-\theta})$.
\end{proof}


\subsubsection{A variation based on $p(z)=p_{odd}(z) \, p(z^2)$} \label{sect odd}
With 
$$
p_{odd}(z) := \prod_{i \mbox{ odd}}(1-z^i)^{-1},
$$
Euler's identity $d(z) = p_{odd}(z)$ suggests a variation on the algorithm of section \ref{sect d(z)}.
It is arguable whether the original algorithm, based on $p(z)=d(z) \, p(z^2)$, and the variant, based on 
$p(z)=p_{odd}(z) \, p(z^2)$, are genuinely different.

Arguing the variant algorithm is different:  the initial 
proposal is $A=(Z_1(x),Z_3(x),Z_5(x),\ldots)$.  Upon acceptance, we have determined $(C_1(\lambda),C_3(\lambda),C_5(\lambda),\ldots)$,
where $\lambda$ is the partition of $n$ that the full recursive algorithm will determine, and $C_i(\lambda)$ is the number of parts of size $i$ in $\lambda$.  The $B$ task will find $(C_2(\lambda),C_4(\lambda),C_6(\lambda),\ldots)$ by iterating the divide-and-conquer idea, so that the second time through the A procedure determines  $(C_2(\lambda),C_6(\lambda),C_{10}(\lambda),\ldots)$,  and the third time through the A procedure determines  $(C_4(\lambda),C_{12}(\lambda),C_{20}(\lambda),\ldots)$,
and so on.

Arguing that the variant algorithm is \emph{essentially} the same:  just as in Euler's bijective proof that $p_{odd}(z)=d(z)$, the
original algorithm had a proposal $A=(\epsilon_1(x),\epsilon_2(x),\ldots)$, which can be used to construct the proposal 
$(Z_1(x),Z_3(x),Z_5(x),\ldots)$  for the variant algorithm.  That is, one can check that starting with independent $\epsilon(i,x) \equiv \epsilon_i(x)$ given by \eqref{def epsilon}, for $j=1,3,5,\ldots,$ $Z_j := \sum_{m \ge 0} \epsilon(j \, 2^m,x) \ 2^m$ indeed has the distribution
of $Z_j(x)$ specified by \eqref{dist Z}, with $Z_1,Z_3,\ldots$ independent.  And conversely, one can check that starting with the
independent geometrically distributed $Z_1(x),Z_3(x),\ldots$, taking base 2 expansions yields independent 
$\epsilon_1(x),\epsilon_2(x),\ldots$ with the Bernoulli distributions specified by \eqref{def epsilon}.
Hence one \emph{could} program the two algorithms so that they are coupled:  starting with the same seed, they would produce the same $T_A$ for the initial proposal, and the same count of rejections before the acceptance for the first time through the A procedure, with same $T_A$ for that first acceptance, and so on, including the same number of iterations before finishing.   Under this coupling, the original algorithm produces a partition $\mu$ of $n$, the variant algorithm produces a partition $\lambda$ of $n$ --- and we have implicitly defined the deterministic bijection $f$ with $\lambda=f(\mu)$.

Back to arguing that the algorithms are different:  we believe that the coupling described in the preceding paragraph
supplies rigorous proofs for the  analogs of Theorems \ref{recursive theorem} and \ref{prop root 2}.  For Theorem 
\ref{theorem sum of series} however, one should also consider the computational cost of Euler's bijection, for various costing schemes, and we propose the following analog, for the variant based on $p(z)=p_{odd}(z) \, p(z^2)$, combined with the trick of moving $\epsilon_1(x)$ from the $A$ side to  the $B$ side, as in Section \ref{sect parity}:

\begin{theorem}\label{theorem sum of series variation}
\emph{Assume} that the  cost $D(n)$  to propose $(Z_1(x),Z_2(x),\ldots, Z_n(x))$, with $x=x(n)$, satisfies $D(n) =  n^\theta$ times a slowly varying function of $n$.  
Assume \emph{also} that the  cost $D_A(n)$ to propose  
$A = (Z_1(x)-\epsilon_1(x),Z_3(x),Z_5(x),\ldots)$ satisfies 
$D_A(n) \sim D(n)/2$.

Then, the asymptotic cost of the 
entire algorithm is
$$ \frac{1}{1-1/4^\theta} \,  \sqrt{2}\, D(n)/2 \le   \sqrt{2}\, D(n) .
$$
\end{theorem}
 \begin{proof}
Essentially the same as the proof of Theorem \ref{theorem sum of series}.
\end{proof}
It is plausible that the cost function $C(n)$ from Theorem \ref{theorem sum of series} and the the cost function $D(n)$ from Theorem \ref{theorem sum of series variation} are related by $C(n) \sim D(n)$;  note that this depends on the choice of costing scheme, essentially asking whether or not the algorithmic cost of carrying out Euler's odd-distinct bijection is
negligible.

\subsection{Complexity considerations}
\label{sect cost}
In Remark~\ref{X} and 
at the end of Section \ref{sect acceptance} we note that in the general view of probabilistic divide-and-conquer algorithms,
a key consideration is computability of the acceptance threshold $t(a)$.  
Recall that we write $p(n)$ for the number of integer partitions of size~$n$. 
The case of integer partitions, using any of  the divide-and-conquer algorithms  of Section \ref{sect d(z)}, is perhaps exceptionally easy, in that computing the acceptance threshold 
is essentially the same as evaluating $p(n)$, an extremely well-studied task.


For $n>10^4,$ a \emph{single term} of the Hardy--Ramanujan asymptotic series suffices to evaluate
$p(n)$ with relative error less than $10^{-16}$;  see  Lehmer \cite{Lehmer38,Lehmer39}.\footnote{We thank David Moews for bringing these papers to our attention.} 
This single term is
\begin{equation} \label{hr1}
hr_1(n) := \frac{\exp(y)}{4 \sqrt{3} (n-\frac{1}{24})} \ \left(1 - \frac{1}{y} \right),   \mbox{ where } y= 2c\sqrt{n-\frac{1}{24}} ,
\end{equation}
and numerical tabulation\footnote{done in Mathematica\textsuperscript{\textregistered} 8.},  together with Lehmer's guarantee, shows that
\begin{equation} \label{sufficiently large HR}
 \| \ln p(n)  - \ln hr_1(n) \|  < 10^{-16}  \mbox{ for all } n \ge 489. 
 \end{equation}

While Equation~\eqref{sufficiently large HR} is sufficient for most practical aspects of computing, with probability $2^{-k}$ we will need more than $k$ bits of $p(n)$ in order to determine the rejection step, with $k$ at most $\lceil\log_2(p(n))\rceil$. 
In the algorithm to sample a partition of size $n$,  we start with \eqref{dist Z}   with the choice $x=x(n) = e^{-c / \sqrt{n}}$;  the normalizing constant is  then $c(x) := \prod_{i \ge 1} (1-x^i)$, and \eqref{prob T=n} is explicitly $\p_x(T=j) = p(j)\, x^j \prod_{i \ge 1} (1-x^i)  $.
After taking into account the factors of 2 in \eqref{duck not 2}, and writing $y=x^2$, the threshold function  $t(\cdot) \equiv t_n(\cdot)$ 
for Algorithm~\ref{IP SS PDC procedure} is given by
\begin{equation}\label{super fraction}
 t(\ell) = \frac{ \p_y(T = \ell)}{\max_j \p_y(T = j)} = \frac{ p(\ell)\, y^\ell }{\max_j p(j)\, y^j }, 
\end{equation}
where the rigthmost expression is the result of cancellation of the factor $c(y)$ from the middle expression.
Let us first focus on the numerator on the far right side of \eqref{super fraction}. 
The cost to evaluate $p(n)$ is not so straightforward, however, and an approach is outlined in~\cite{johansson2012efficient} which is $O(n^{1/2 + o(1)})$. 
If one wants an arithmetic cost bound that is better than $O(n^{1/2+o(1)})$, then the following is needed.

\begin{lemma}\label{r bits}
The arithmetic cost to obtain the first $r$ bits of $p(n)$ is $O(r\, \log^5(n) ).$ 
\end{lemma}
\begin{proof}
Our proof is a modification of the proof in Section~3 of~\cite{johansson2012efficient}, which assumes that all terms in the Rademacher series expansion~\cite{Rademacher} are evaluated with an absolute error of at most $2^{-3} / N$, where $N$ is the number of terms of the Rademacher series taken. 
Let $s := \lfloor \log_2 p(n)\rfloor$ denote a lower bound on the number of nontrivial bits in $p(n)$.  
In our case, since the main contribution of error comes from truncating the series, and this error, in the notation of Equation~\eqref{RnN}, is $R(n,N)$, we choose an $N$ which bounds $R(n,N)$ by $2^{s-r-1}$, and hence only require each term of the series to be evaluated with an absolute error of $2^{s-r-4} / N$, which is much less restrictive. 

The details are straightforward and messy, and the punchline is that when $r = o(\sqrt{n})$, and even in fact for $r$ up to a small constant times $\sqrt{n}$, the first two terms of the Rademacher series to approximate $p(n)$ suffice\footnote{Empirically the first term suffices, but the error analysis in~\cite{Rademacher, Lehmer39} requires a second term.}.  
Thus, we require that each of the terms in the Rademacher series are evaluated with an absolute error of $2^{s-r-5}$ or less. 
\ignore{Also, instead of requiring that the $k^{th}$ term in the Rademacher series be evaluated at a precision of $O(n^{1/2}/k)$ bits, since it was used with $N = O(\sqrt{n})$ terms, and in this case we simply need at most $N = 2$ terms, each at a precision of $r$ bits; this is because the maximum error with $N=2$ terms is bounded by $R(n,N)$, with $\log_2 R(n,N) = \log_2(e) \pi \sqrt{2/3} \sqrt{n}/2$ plus a logarithmic correction term. 
}
Then one substitutes this into \cite[Equation~3.13]{johansson2012efficient}, and obtains the cost to evaluate the leading $r$ bits of $p(n)$ is at most 
\[ O\left(\sum_{k=1}^{2} \log k\, M\left(r\right) \log^2\left(r\right)\right)= O(r\, \log^{4}(n)), \]
where $O(M(r) \log^2(r))$ is the cost to evaluate the largest $r$ bits of an elementary function\footnote{It is noted in \cite{johansson2012efficient} that there are two competing algorithms to evaluate the largest $r$ bits of elementary functions, one which is $O(M(r) \log(r))$, and another which is $O(M(r)\log^2(r))$, but is faster in practice.  Since logarithmic terms do not affect the conclusion of Lemma~\ref{decision}, we are content with the the slightly larger estimate.}, and $M(r) = O(r \log^{1+o(1)}(r))$ is the cost to multiply two $r$-bit numbers. 
When $r = \Omega(\sqrt{n})$, we may simply evaluate $p(n)$ exactly, which was shown to require $O(n^{1/2} \log^{4+o(1)}(n))$ arithmetic operations in \cite{johansson2012efficient}, and hence we have opted for the simpler and concrete bound $5$ in place of $4 + o(1)$ since extra factors of $\log(n)$ do not affect Lemma~\ref{decision}.

The details are as follows.
Let $\mu = \frac{\pi}{6} (24n-1)^{1/2}$.  The Rademacher series for $p(n)$, as given by \cite{Lehmer38},  is
\[ p(n) = \frac{\sqrt{12}}{24n-1} \sum_{k=1}^N \frac{A_k(n)}{k^{1/2}} \left( \left(1 - \frac{k}{\mu}\right)e^{\mu / k} + \left(1 + \frac{k}{\mu}\right)e^{-\mu/k}\right) + R(n,N), \]
where 
$A_k(n)$ is a complicated sum of $24k$th roots of unity, and
\begin{equation}\label{RnN}
 R(n,N) = \sum_{k=N+1}^\infty \frac{A_k(n)}{k^{1/2}}\left( \left(1 - \frac{k}{\mu}\right)e^{\mu / k} + \left(1 + \frac{k}{\mu}\right)e^{-\mu/k}\right). 
 \end{equation}
Let $v := \mu / N$, and define
\[ F(n,N) := \frac{N^{-2/3}\pi^2}{\sqrt{3}}\left(\frac{\sinh(v)}{v^3} + \frac{1}{6} - \frac{1}{v^2}\right). \]
It was shown in~\cite{Lehmer38} that
\[ |R(n,N)| < F(n,N) \]
for all positive integers $n$ and $N$. 

Let $t_k$ denote the $k^{th}$ term in the Rademacher series expansion, and suppose $\widehat{t}_k$ is the numerical approximation stored in a computer. 
In order to compute the leading $r$ bits of $p(n),$ 
we need
\begin{equation}\label{numerical} | t_k - \widehat{t}_k | < \frac{0.125\times 2^{s-r}}{N}. \end{equation}
This differs from~\cite[Equation~(3.1)]{johansson2012efficient} by the extra factor of $2^{s-r}$, which means we can be less precise when computing the value of $t_k$. 
A theorem analogous to~\cite[Theorem 4]{johansson2012efficient} thus holds; i.e., for $n > 2000$, for Equation~\eqref{numerical} to hold, it is sufficient to evaluate $t_k$ using a precision of $b = \max( \log_2(N)  + \log_2(|t_k|)  - (s- r) + \log_2(n) + 3, \frac{1}{2} \log_2(n)+5, 11)$ bits. 

There are two cases to consider: $r = o(\sqrt{n})$ and $r = \Omega(\sqrt{n})$.  
When $r = \Omega(\sqrt{n})$, we may simply evaluate $p(n)$ exactly, which requires $O(n^{1/2} \log^{4+o(1)}(n))$ arithmetic operations. 
When $r = o(\sqrt{n})$, however, we note that $\sum_{i=1}^N t_k$ has an absolute error of at most $F(n,N)$, which is at most $e^{\mu / N}$/2. 
Thus, $t_1+t_2$ guarantees \emph{at least} $\log_2(e) \mu/2$ correct bits, up to a logarithmic correction factor which can be made explicit. 

We then make the assumption, as is done in~\cite[Section~3.2]{johansson2012efficient}, that $r$-bit floating point numbers can be multiplied in time $M(r) = O(r \log^{1+o(1)}r)$ time. 
Then, to calculate the partial sum consisting of $2$ terms, we evaluate 
\[ O\left(\sum_{k=1}^{2} (\log k)\,  M\left(r\right) \log^2\left(r\right)\right)= O(r\, \log^{4}(n)), \]
where $M(r) = O(r \log^{1+o(1)}(r))$ is assumed to be the cost to multiply two $r$-bit numbers. 
\end{proof}

\begin{lemma}\label{decision}
For random $U$ distributed uniformly in $(0,1)$, the expected number of arithmetic operations for determining whether $U \leq t(\ell),$ where $t(\ell)$ is given in Equation~\eqref{super fraction}, is $O(\log^{5}(n)).$
\end{lemma}
\begin{proof}

The probability of needing more than $r$ bits is given by $2^{-r}$, and so the total expected amount of arithmetic operations for the evaluation of $t(\ell)$ is given by 
\[  O\left(\sum_{k \geq 1} k\, \log^{5}(n)\, 2^{-k}\right) = O(\log^{5}(n)). \]
\end{proof}

To deal with the denominator of \eqref{super fraction},  there are two strategies.   The first strategy is to find
the argument $j$ which achieves $\max_j p(j)\, y^j$, i.e., which achieves $\max_j ( j \log y + \log p(j))$.  Now, 
\cite{DeSalvoPak,Nicolas} shows that $n \mapsto p(n)$ is log-concave for all $n>25$,  so bisection search over the range $j=1$ to $n$  requires $O(\log n)$ evaluations of $p(j)$.  To be very careful, we must also note that we are dealing with \emph{approximate} evaluations of $p(j)$;  since
$p(k) - p(k-1) \sim p(k) / \sqrt{k}$, we need only calculate the first $\log(n)$ bits of two terms of the form
$j \log x + \log p(j)$ in order to determine which is greater, and Lemma \ref{r bits} applies.   
The second strategy is to modify the threshold function $t(\ell)$ in \eqref{super fraction},  by giving away a factor
of the form $1+\delta$ for some fixed $\delta>0$;   this  has the effect of increasing the number of rejections, and hence the number of random bits used, by this same factor.  In detail,   \eqref{sufficiently large HR} implies that
$c_0 := \max_{j \ge 1} | \ln p(j)  - \ln hr_1(j) | < \infty$, and one can easily evaluate $c_0$.  Hence, with $1+\delta := \exp(c_0)$, we have $p(j) \le (1+\delta) hr_1(j)$ for all $j \ge 1$, so the modified rejection threshold,
$t'(\ell) := p(\ell) y^\ell/ \max_j ( (1+\delta) hr_1(j) y^j)$,  satisfies $t' \le 1$, and is a valid rejection threshold. 

\ignore{
To obtain the denominator, we use the fact that with $x = e^{-c / \sqrt{n}}$, $T$ is asymptotically normally distributed with expected value $n$, and hence is concentrated around its mean.  Thus, for some positive integer $n'$, we have $\max_k p(k) x^k = p(n')x^{n'}$, and we now show that $|n' - n| = O(1)$. 

\begin{lemma}\label{concentrated}
Let $x = e^{-c / \sqrt{n}}$, where $c = \pi/\sqrt{6}$.  Then the value $n'$, for which $\max_k p(k) x^k = p(n')x^{n'}$, satisfies $|n - n'| = O(1)$ asymptotically as $n$ tends to infinity. 
\end{lemma}

For each $n$, let $b_{k} := p(k)e^{-c\, k/\sqrt{n}}$. 
A search algorithm would then start by calculating $b_{n}, b_{n-1}, b_{n+1}$.  
If $b_{n}$ is the max of all three, then we stop. 
If not, then if $b_{n-1}$ is larger, we continue to calculate $b_{n-k}$, $k=1, 2, \ldots$ until a local max is found; similarly, if $b_{n+1}$ is larger than $b_{n}$, then we continue to calculate $b_{n+k}$, $k=1,2,\ldots$ until a local max is found. 
In order to guarantee the local minimum found is a global maximum, we use the fact that $p(n)$ is log-concave for all $n>25$~\cite{DeSalvoPak}, hence $b_n$ is also log--concave. 
Thus, for $n>25$, the algorithm defined above converges to \emph{the} global max of $b_{k}$, say $b_{n'}$, in time that is $O(n -n')$, where $n$ is the initial guess. 
Thus, by Lemma~\ref{concentrated} above, we find the exact maximum after $O(1)$ steps. 


The final ingredient is that we do not need to calculate $b_k$ exactly; rather, since $p(k) - p(k-1) \sim p(k) / \sqrt{k}$, we need only calculate the first $\log(k)$ bits of the expressions in order to determine which is greater. 
Thus, by Lemma~\ref{r bits}, the total cost to perform this search algorithm is at most $O(\log^\alpha(n))$, where $\alpha$ is some fixed, universal constant independent of all parameters. 
}    

\subsection{Partitions with restrictions}
\ignore{
As with unrestricted partitions, if $n$ is moderate and a recursive formula exists, analogous to that of 
Section \ref{sect table}, 
then the table method is the most rapid, and divide-and-conquer is not needed.  However, the requirement of random access storage of size $n^2$ is a severe limitation.
}
The self-similar recursive divide-and-conquer method of Section \ref{sect d(z)} is nearly unbeatable for large $n$, for unrestricted partitions.
There are many classes of partitions with restrictions that iterate nicely, and should be susceptible to a corresponding recursive divide-and-conquer algorithm, \emph{provided} efficient enumeration, analogous to \eqref{hr1}, is available.  Some of these classes, with their self-similar divisions, are
\begin{enumerate}
\item distinct parts,  $d(z) = d_{odd}(z)d(z^2) $;
\item odd parts, $p_{odd}(z) = d_{odd}(z)p_{odd}(z^2) $;
\item distinct odd parts, $d_{odd}(z) = d_{*}(z) d_{odd}(z^3)$.
\end{enumerate}
Here $d_{\ast}(z) = (1+z)(1+z^5)(1+z^7)(1+z^{11}) \cdots$ represents distinct parts $\equiv \pm 1 \mod 6$.  Other recurrences are discussed in \cite{recurrenceasymptotics,Pak,Remmel}, and the standard text on partitions \cite{AndrewsP}.

It is straightforward to apply deterministic second half to restricted partitions, as the rejection probability is given explicitly by a ratio of geometric probabilities. 
To apply recursive PDC to restricted partitions, one must have some means 
of explicitly bounding the error in any approximation to the number of partitions subject to restrictions, in order to evaluate the functional $\mathbbm{1}(U < t(a))$. 
As Section~\ref{sect cost} shows, this is accessible, via the work of Rademacher \cite{Rademacher} and Lehmer \cite{Lehmer38, Lehmer39}, when there are no restrictions. 
When the restrictions are of a certain form, then there are completely effective error rates available~\cite{Sills}. 
Other generalizations to Rademacher's method are also available~\cite{almkvist1991, almkvist, Laughlin, Engel}. 

\ignore{
It is not easy to come up with examples where the optimal divide-and-conquer 
is like that in  Section \ref{sect divide}, based on small parts versus large parts.  
One suggestion is partitions with all parts prime; there should be a large range of $n$ for which table methods are ruled out by the memory requirement, while the $n$ memory, $b \times n$ computational time to calculate rejection probabilities is not prohibitive.  Another suggestion is partitions with a restriction 
on the multiplicity, for example, a  part of size $i$ can occur at most $f(i)$ times --- with $f$ sufficiently complicated as to rule out
recursive formulas such as those in the preceding paragraph.  
 }

\ignore{
Note, we have taken $D(n)$ to be the cost of proposing \emph{all} of $(Z_1(x),Z_2(x),\ldots, Z_n(x))$, so that the top level $A$ proposal uses half of that cost; but the sum of the geometric series for carrying out the main problem and all its subproblems
brings in an additional factor of two; the one half cancels with the two.
}

\ignore{
One of the motivations for wanting an honest  \emph{sample}, that is, observations $S_1,S_2,\ldots$ which not only are distributed exactly as the pre-specified $S$, but are also independent, is to be able to estimate properties involving more than one $S$ at a time.  A natural example, discussed in Section \ref{sect missing}, is dominance between independently chosen \emph{pairs} of partitions of $n$.  One might want an honest sample of $m$ partitions, to form $m/2$ independent pairs, so that the variability of the estimator is that of the Binomial($m/2,\pi_n$) distribution.  An alternate is the following unbiased estimator, whose variability cannot be analyzed.  We start as in \eqref{def all count}, but take instead $W_2$ pairs, where $W_2$ is 
\begin{equation}\label{def all 2 count}
   W_2 = W_2(m_1,m_2) = \sum_{1 \le i \ne i' \le m_1,1 \le j \ne j' \le m_2} I_{i,j} I_{i',j'} ,
\end{equation}
where we write $I_{ij}$ for the indicator $h(A_i,B_j)$.
We test for dominance in each of the $W_2$ pairs.
Then, analogous to a special case of Theorem \ref{unbiased}, conditional on $(W_2>0)$, the proportion of pairs exhibiting dominance is an unbiased estimator, with
$$
\e\left(\left.  \frac{1}{W_2} \sum_{1 \le i \ne i' \le m_1,1 \le j \ne j' \le m_2}   I_{i,j} I_{i',j'}1(\lambda \mbox{ dominates} \lambda') \right| W_2  > 0\right)\ = \pi_n.
$$
Here of course $\lambda$ is the partition of $n$ formed by combining $A_i$ with $B_j$, and $\lambda'$, rather than the dual of $\lambda$,  is the partition of $n$ formed by combining $A_{i'}$ with $B_{j'}$.  The mutual independence of the four
random elements, $A_i, A_{i'},B_j$, and $B_{j'}$, is all that is needed to have $\lambda,\lambda'$ independent, so that conditioning on $ I_{i,j} I_{i',j'}=1$, we have two independent uniformly chosen partitions.
}

 \section{Probabilistic divide-and-conquer with mix-and-match}\label{mam}

\subsection{Algorithmic implications of the basic lemma}\label{subsection basic}

Assume that one wants a sample of fixed size $m$ from the distribution of $S$.  That is, one wants to carry out a simulation that provides $S_1,S_2,\ldots$, with $S_1,S_2,\ldots,S_m$ being  independent, with each equal in distribution to $S$.  According to Lemma
\ref{lemma 1}, this can be done by providing $m$ independent pairs $(X_i,Y_i)$, $i=1$ to $m$, each equal in distribution to $S$.

A reasonable choice of how to carry this out (not using mix-and-match) is given in Algorithm~\ref{PDC pre-mix procedure}.

\ignore{
\vspace{1pc}
\noindent {\bf Outline of an algorithm to gather sample of size $m$.}\\
\noindent {\bf Stage 1}.  Sample $X_1,X_2,\ldots,X_m$ from the distribution of $X$, i.e., from $\L(X) = \L( \, A \, | \, h(A,B)=1 \, )$.\\
\noindent {\bf Stage 2}.  Conditional on the result of stage 1 producing $(X_1,\ldots,X_m)=(a_1,\ldots,a_m)$,
find $Y_1,Y_2,\ldots,Y_m$,  independent, with distributions $\L(Y_i) = \L( \, B \, | \, h(a_i,B)=1 \, )$ for $i=1,2,\ldots,m$.
}

\begin{algorithm}{\rm
\begin{algorithmic}
\State 1.  Generate $X_1, X_2, \ldots, X_m$ i.i.d.~from $\L(A\, |\, h(A,B)=1)$.
\State 2. Conditional on $(X_1,\ldots,X_m)=(x_1,\ldots,x_m)$, \\
\ \ \ \ generate $Y_1,Y_2,\ldots,Y_m$,  independent, with distributions \\
\ \ \ \ $\L(Y_i) = \L( \, B \, | \, h(x_i,B)=1 \, )$ for $i=1,2,\ldots,m$.
\State 3.  Return $( (X_1, Y_1), (X_2, Y_{2}), \ldots, (X_m, Y_{m}) )$.  
\end{algorithmic}}
\caption{Probabilistic Divide-and-Conquer $m$ samples} 
\label{PDC pre-mix procedure}
\end{algorithm}

\vspace{1 pc}
  Note that in general, conditional on the result of stage 1, the $Y_i$ in stage 2 are \emph{not identically} distributed.

\subsection{\emph{Simple} matching enables mix-and-match}\label{sect match}

An important class of PDC sampling algorithms are those for which the matching condition $h(A,B) = 1$ defines a disjoint union of complete bipartite graphs; that is, there exists a (set) partition $I_1, I_2, \ldots$ of $\mathcal{A}$ and a corresponding (set) partition $J_1, J_2, \ldots$ of $\mathcal{B}$ such that $h(A,B) = 1$ if and only if $A\in I_i$ and $B \in J_i$, $i=1,2,\ldots$.  
When this condition holds, we call the matching \emph{simple}, and say that \emph{mix-and-match} is enabled. 

The following lemma serves 
to clarify
the logical structure of what is needed to enable 
mix-and-match.

\begin{lemma}\label{match lemma}
Given $h : \A \times \B \to \{0,1\}$, the following two conditions are equivalent: \\
Condition 1: 
$\forall a,a' \in \A, b,b' \in \B,$
\begin{equation}\label{square} 1=h(a,b)=h(a',b)=h(a,b') \mbox{ implies } h(a',b')=1,
\end{equation}
Condition 2: 
$\exists \C$, and functions  $c_\A: \A \to \C, c_\B:\B \to \C$, so that  $\forall (a,b) \in \A \times \B,$
\begin{equation}\label{color} 
h(a,b)=1(c_\A(a)=c_\B(b)).
\end{equation}
\end{lemma}\qed

We think of $\C$ as a set of \emph{colors}, so that condition \eqref{color} says that $a$ and $b$ match
if and only if they have the same color.

\ignore{
\begin{proof}
That \eqref{color} implies \eqref{square} is trivial.  In the other direction, it is easy to check that
\eqref{square} implies that the relation $\sim_\A$ on $\A$ given by $a \sim_\A a'$ iff $\exists b \in \B, 1=h(a,b)=h(a',b)$
is an \emph{equivalence} relation.  Likewise for
the relation $\sim_\B$ on $\B$ given by $b \sim_\B b'$ iff $\exists a \in \A, 1=h(a,b)=h(a,b')$.
For the set of colors, $\C$, we might take either the set of equivalence classes of $\A$ modulo $ \sim_\A$, or the set of 
equivalence classes of 
 $\B$ modulo $\sim_\B$, and then \eqref{square} also provides a bijection between these two sets of equivalence classes, to induce \eqref{color}.
\ignore{Also, we do not focus on \emph{measurability}, but we note that under the assumption that $\A$ and $\B$ are sets endowed with sigma-algebras of measurable sets, and $h$ is measurable with respect to the product 
sigma-algebra, the the color maps $c_A$ and $c_B$ may be taken to be measurable.
} 
\end{proof}

\begin{remark}
 The proof of Lemma \ref{match lemma}, with equivalence classes of
$\A /\! \!\sim_\A$ and $\B / \!\! \sim_\B$, shows that the pair of coloring functions satisfying
\eqref{color} is  \emph{essentially unique}.  Specifically, unique apart from relabeling and padding, i.e.,  an arbitrary permutation on the names of the colors used, and enlarging the range, $\C$, to an arbitrary superset of the image.\\
\end{remark}
} 

\ignore{
\begin{remark}
The statement of Lemma \ref{match lemma} shows that coloring is   \emph{not essentially} an issue of
 \emph{sufficient statistics}.  After all,  hypothesis \eqref{square} only concerns the logical
structure of the matching function $h$ appearing in \eqref{def h}, and does not involve the distributions
on $\A$ and $\B$ appearing in \eqref{def AB}.\\
\end{remark}
} 
\begin{remark}
When \eqref{color} holds, we can write the event that $A$ matches $B$ as a union indexed by the color involved:
$$
 \{ h(A,B)=1\} = \cup_{k \in \C}  \{  c_\A(A)=k, c_\B(B)=k\},
$$
so that $p = \sum_{k \in \C}  \p(c_\A(A)=k, c_B(B)=k)$, and we see that at most a countable set of colors $k$ contribute a strictly positive amount to $p$.  As a notational convenience, we take $\mathbb{N} \subset \C$, and use positive integers $k$ for the names of colors that have 
\begin{equation}\label{possible color}
 \p(c_\A(A)=k, c_B(B)=k) > 0.
\end{equation}
\end{remark}

\begin{remark} Here is an example which is \emph{not} simple, i.e., does not satisfy \eqref{color}.  The configuration method for random $r$-regular graphs on $n$ vertices,
\cite[Section 2.4]{randomgraphs},  involves a uniform choice over a set of size $(2n-1)!!$.  For any choice $1 < b < n$, one might take $\A,\B$ with $| \A \times \B|=(2n-1)!!, |\B|=(2b-1)!!$, so that $A$ corresponds to the first $n-b$ choices that need to be made to specify a configuration, and $B$ corresponds to the final $b$ choices.  The matching function $h$ is given by $h(A,B)$, the indicator that the multigraph implied by the configuration $(A,B)$ has no loops and no multiple edges.
\end{remark}

\begin{remark}
Observe that the matching function from Section~\ref{Algorithms Section} with $h(A,B) = 1(T_A +T_B = n)$, for $(A,B) \in \A \times \B$, satisfies condition 2 of Lemma \ref{match lemma}, with $c_\A(A)=T_A$ and $c_\B(B)=n-T_B$.  Hence \emph{mix-and-match} is enabled.  
\end{remark}

The intent of the following lemma is to show that
  if $h$ satisfies \eqref{color}, then mix-and-match  strategies can be used in stage 2 of the broad outline of
Section \ref{subsection basic}. 
We consider a procedure which proposes a sequence $D_1,D_2,\ldots$ of elements of $\B$ with the following properties:
we assume there is a sequence of $\sigma-$algebras  $\F_0 \subset \F_1 \subset \F_2 \subset \cdots$ for which $D_n$ is $\F_n$ measurable.  We think of $\F_0$ as carrying the information from stage 1 of an algorithm along the lines described in Section \ref{subsection basic}, carrying information such as
``which demands $a_1,a_2,\ldots$ must be met'' --- or reduced information, such as the colors $c_\A(a_1),\ldots,c_\A(a_m)$.

\begin{lemma}\label{match lemma 2}
Assume that $h$ satisfies \eqref{color} and that the following conditions hold:
\begin{enumerate}
\item For every $n \ge 1$ and $k$ satisying \eqref{possible color}, conditional on $\F_{n-1}$ together with $c_\B(D_n)=k$, the distribution of $D_n$ is equal to $\L(B | c_\B(B)=k)$.

\item For every $k$ satisying \eqref{possible color}, 
\begin{equation}\label{infinitely many}
\mbox{with probability 1, infinitely many }n \mbox{ have } c_\B(D_n)=k.
\end{equation}

\item Define stopping times $\tau^{(k)}_i$, the ``time $n$ of the $i$-th instance of $c_\B(D_n)=k$", by $\tau^{(k)}_0=0$
and for $i \ge 1, \tau^{(k)}_i= \inf\{n> \tau^{(k)}_{i-1}: c_\B(D_n)=k\}$.  We write $D(n) \equiv D_n$, to avoid multi-level
subscripting, and define $B^{(k)}_i := D(\tau^{(k)}_i)$ for $i=1,2,\ldots$.
\end{enumerate}

\emph{Then}, for each $k$, the  $B^{(k)}_1,B^{(k)}_2,\ldots$ are independent, with the distribution \mbox{$\L(B | c_\B(B)=k)$}, and as $k$ varies, these sequences are  independent.

\end{lemma}

\begin{proof}
The proof is a routine exercise 
by summing over all possible  values for the random times $\tau^{(k)}_i $.
Writing out the  full argument would be  notationally messy, and not interesting.
 \end{proof}

\ignore{
\subsection{Use of acceptance/rejection sampling}\label{sect acceptance mix}
Assume that we know how to simulate $A$ --- this is under the distribution in \eqref{def AB}, where $A,B$ are independent.
But, we need instead to sample from an alternate distribution, denoted above as that of $X \in \A$. 
The rejection method recipe, for using \eqref{A to X}, may be viewed as having 2 steps, as follows.

\begin{enumerate}

\item Find a threshold function $t: \A \to [0,1]$, with $t(a)$ proportional to $\e h(a,B)/p$, i.e., $t$ of the form
$t(a) = C\times \e h(a,B)/p$ for some positive constant $C$.

\item Propose i.i.d.~samples $A_1,A_2,\ldots$ and independently generate uniform (0,1) random variables $U_1,U_2,\ldots$, until $U_i \leq t(A_i).$

\end{enumerate}
See \cite{DeVroye} for a thorough treatment of acceptance/rejection methods.  \\
}

\subsection{Divide-and-conquer with mix-and-match} \label{sect mix}
 
 
When a sample of size $m>1$ is desired, and the matching function is simple, Lemmas~\ref{lemma 1},~\ref{match lemma},~and~\ref{match lemma 2} combine to suggest the following algorithm.  
Stage 1 of the algorithm generates samples $A_1,A_2,\ldots, A_m$ from $X$, according to Lemma \ref{lemma 1}.   This creates a multiset of $m$ colors, $\{c_1,\ldots,c_m\}$, where $c_i=c_\A(A_i)$.  We think of these as $m$ demands that must be met by Stage 2 of the algorithm.   One strategy for Stage 2 is to generate an i.i.d.~sequence of
samples of $B$;  initially, for each sample, we compute its color $c=c_\B(B)$ and check whether $c$ is in the multiset of demands
$\{c_1,\ldots,c_m\}$;  when we find a match, we pair $B$ with one of the $A_i$ of the matching color, to produce our first sample, which we set as  $S_i=(A_i,B)$.  Then we reduce the multiset of demands by one element, and iterate, until all $m$ demands have been met.
Lemma \ref{match lemma 2} implies that the resulting list $(S_1,S_2,\ldots,S_m)$ is an i.i.d.~sample of $m$ values of $S$, as desired.

A simple algorithmic view of the mix-and-match PDC algorithm is as follows, which assumes that mix-and-match is enabled.

\begin{algorithm}{\rm
\begin{algorithmic}
\State 1.  Generate $X_1, X_2, \ldots, X_m$ i.i.d.~from $\L(A\, |\, h(A,B)=1)$.
\State 2.  Generate $B_1, B_2, \ldots$ i.i.d.~from $\L(B)$ until all $X_i$'s have a match.
\State 3.  Return $( (X_1, B_{j_1}), (X_2, B_{j_2}), \ldots, (X_m, B_{j_m}) )$.  
\end{algorithmic}}
\caption{Probabilistic Divide-and-Conquer Mix-and-Match} 
\label{PDC mix procedure}
\end{algorithm}

We refer to Step 1 as the ``$A$ phase" and Step 2 as the ``$B$ phase."  
The $A$ phase can be performed using the techniques of previous sections.  The $B$ phase is straightforward since we are sampling from the \emph{unconditional} distribution $\L(B)$.  

\begin{remark}\label{order matters}
For Step 3, it is important to note that  \emph{the order in which completed samples are returned is not arbitrary!}  
In particular:
\begin{enumerate}
\item[(a)] We cannot report the sample in the order in which matching $B_{j_i}$'s are found.  That is, if we denote by $(j_1), (j_2), \ldots, (j_m)$ the ordering of the sequence $j_1, \ldots, j_m$ from smallest to largest, and $i_1, \ldots, i_m$ the corresponding indices for the matching samples from the $A$ phase, then in general 
\[( (X_{1_1}, B_{(j_1)}), (X_{i_2}, B_{(j_2)}), \ldots, (X_{i_m}, B_{(j_m)}) )
\] 
is \emph{not} an i.i.d.~sample from $\L(S)$.
\item[(b)] Continuing with the previous point, we can either apply a random permutation to the list, or report the sample in the order in which the $X_i$'s are generated in the $A$ phase, as we have done in Step 3.  
\end{enumerate}

The practical significance of item (a) above is that if we are running a simulation for $m$ samples, and we wish to abort the simulation early and/or print out partial results for a smaller sample size $r \leq m$, \emph{we can only certify that $( (X_1, B_{j_1}), \ldots, (X_r, B_{j_r}) )$ is an i.i.d.~sample of size $r$ from $\L(S)$ if every $X_i$, $i \leq r,$ has been matched with a $B_{j_i}$.}
\end{remark}

\ignore{
\begin{remark}\label{remark non sample}
Note that in the above, if the first match found, $(A_i,B)$, is labelled as $S_1$, and the second matching pair is labelled $S_2$, and so on, then the resulting list $(S_1,S_2,\ldots,S_m)$ is \emph{not} necessarily an i.i.d.~sample of $S$;  the colors $c$ with $\p(c=c_\B(B))$ large would tend to show up earlier in the list.  
\end{remark}
} 

 \subsubsection{Roaming $x$}
\label{roaming x}

Consider again a sample of size $m=1$.  Having accepted $X=(Z_{b+1},\ldots,Z_n)$ with color $k=T_A$, in the notation 
of \eqref{def TAB}, we now need $Y$, which is $B=(Z_1,\ldots,Z_b)$ conditional on having color $k$, which simplifies to having $n-k=T_B :=\sum_{i=1}^b i Z_i$.  One obvious strategy is to sample $B$ repeatedly, until getting lucky. The distribution of $B$  is specified by \eqref{dist Z} and \eqref{def x(n)} --- with a choice of parameter, $x=x(n)$, not taking into account the values of $b$ and $n-k$.  A computation similar to \eqref{prob lambda} shows that the distribution of $(Z_1(y),\ldots,Z_b(y))$ conditional on $\left(\sum_{i=1}^b i Z_i(y) =n-k\right)$ is the same, for all choices $y \in (0,1)$.  
 
As observed in \cite[Section 5, page 114)]{IPARCS},  the $y$ which maximizes $\p_y (T_B=n-k)$ 
 is the
solution of $n-k=\sum_{i=1}^b \e i Z_i(y)$.  Thus, in the case $m=1$,  the optimal choice of $y$ is easily prescribed.
However, for large $m$, 
 using mix-and-match brings into play a complicated coupon collector's situation.
   With a multiset of demands $\{c_1,\ldots,c_m\}$ from Section \ref{sect mix}, the freedom to allow $y$ to roam allows us to tilt the distribution in response to the demands that remain at each stage.  The algorithm designer has many choices of global strategy.
Based on computer experiments, it is not obvious whether or not a greedy strategy --- picking $y$  to maximize the chance that  the next proposed $B$ satisfies at least one of the demands --- is optimal.

  \ignore{
\subsubsection{Deliberately missing data}\label{sect missing}
 
In Stage 2 of mix-and-match, starting with $m$ demands $c_1,c_2,\ldots,c_m$, there is a coupon collector's
advantage initially, that wears off when only a few demands remain to be met.


   Suppose we stop before completing the $B$ phase, with some $v$ of the demands remaining to be met.  
     This list of $m-v$ is \emph{not} a sample of size $m-v$, since \emph{sample} requires  i.i.d.~from the original target distribution.  But, had we run the $B$ phase to completion, we would have had an honest sample of size $m$, so there is information in knowing all but $v$ of the $m$ values.
Think of this sample of size $m$ as \emph{the sample},  with some $v$ of its elements being unknown.  For estimates based on the sample proportion, an error of size at most $v/m$ is introduced by the unknown elements. Since the standard deviation, due to sample variability, decays roughly like 
$1/\sqrt{m}$, it makes sense to allow $v$ comparable to $\sqrt{m}$.
}
   \ignore{
A more complicated situation arises when the basic task is to simulate \emph{pairs}
$(\lambda,\mu)$ of partitions, with all $(p(n))^2$ pairs equally likely   --- this arises naturally, when the question is to estimate the probability $\pi_n$ that $\lambda$ dominates $\mu$.\footnote{We say $\lambda$ dominates $\mu$ if
$$
\sum_{i=1}^k \lambda_i \geq \sum_{i=1}^k \mu_i
$$
	for all $k\geq 1$.}
If one has an honest sample of $2m$ partitions of $n$, with $v$  missing items due to terminating the $B$ phase early, then one would have an honest sample of $m-V$ pairs $(\lambda,\mu)$, with random variable $V\le v$.\footnote{The pairing on $\{1,2,\ldots,2m\}$ must be assigned \emph{before} observing which $v$ items are missing.  Pairing up the missing partitions, in order to get $\lceil v/2 \rceil$ missing pairs, in \emph{not valid}; see Remark \ref{remark non sample}.}  
  If we have $v=0$, and $H$ of the pairs contribute a 1 to the estimate of the sample proportion, 
   then the point estimate for $p := \pi_n$ is $H/m$, and the 
standard $(1-\alpha)\%$ confidence interval is
$$ 
\left[\frac{H}{m}-z_{\alpha/2}\sqrt{\frac{p(1-p)}{m}},\frac{H}{m}+z_{\alpha/2}\sqrt{\frac{p(1-p)}{m}}\right].
$$

With $V$ missing pairs, consider the count $K$, how many of the  $m-V$ completed pairs  had  $\lambda$ dominates $\mu$.   We can do a worst-case analysis by assuming, on one side, that all $V$ of the missing pairs
have domination, and on the other side, that none of the $V$ missing pairs have domination;  more succinctly,
$H \in [K,K+V] $.  Hence the confidence interval
$$ 
\left[\frac{K}{m}-z_{\alpha/2}\sqrt{\frac{p(1-p)}{m}},\frac{K+V}{m}+z_{\alpha/2}\sqrt{\frac{p(1-p)}{m}}\right]
$$
is at least a $(1-\alpha)\%$ confidence interval, for the procedure with deliberately missing data\footnote{
We feel compelled to speculate that many users of ``confidence intervals" don't really care about the confidence, nor the width of the interval, but really rely on the center of the interval, which in the standard case is an unbiased estimator. Our interval has center $(K+V/2)/m$, and this value is not an unbiased estimator;  indeed, anything based on $K$, including the natural $K/(m-V)$, is biased in an unknowable way.
}.
} 

 \ignore{
\subsection{For comparison: opportunistic divide-and-conquer with mix-and-match}\label{sect opportunistic}

Recall the setup of \eqref{def AB} -- \eqref{def S}.  
Care was taken in Section \ref{sect match} to enable a mix-and-match procedure that produced a genuine random sample of $(A,B)$ conditional on $(h(A,B)=1)$.  An algorithm that implements the procedure as stated in Lemma \ref{match lemma 2}, produces a certain amount of discarded data, 
both from the rejection threshold of Stage 1 and from the coupon collector's paradigm of matching at Stage 2.  One might try instead generating, for fixed  $m$, independent  $A_1,\ldots,A_m$, $B_1,\ldots,B_m$, and search for \emph{all} matching pairs.  

Of course, the difficulty with the general program ``search all $m^2$ index pairs $(i,j)$'' is that conflicting matches might be found.  Suppose, for example, that exactly two matches are found, with  $A_i$ matching both $B_j$ and $B_{j'}$,  with $j \ne j'$.  It is easy to see that  taking both $AB$ pairs  ruins the i.i.d.~nature of a sample.  Also, though perhaps not as obvious, other strategies, such as suppressing both $AB$ pairs, or taking only the pair indexed by the lexicographically first of $(i,j), (i,j')$, or taking only one pair, based on an additional coin toss, introduce bias relative to  the desired distribution for $S$.
  Here, we point out that nonetheless, the natural opportunistic 
procedure does supply a \emph{consistent} estimator.

Take a deterministic design:  for integers $m_1,m_2 \ge 1$, let $A_1,\ldots, A_{m_1}$ be distributed according to the
distribution of $A$ given in \eqref{def AB}, and let $B_1,\ldots,B_{m_2}$ be distributed according to the distribution of $B$, with
these $m_1+m_2$ sample values being  independent.  The {\bf opportunistic} observations, under a \emph{take-all-you-can-get} strategy, are all the pairs $(A_i,B_j)$ for which $h(A_i,B_j)=1$.  To use these in an estimator, one would naturally count the available pairs, via
\begin{equation}\label{def all count}
   W = W(m_1,m_2) = \sum_{1 \le i \le m_1,1 \le j \le m_2} h(A_i,B_j),
\end{equation}
and for a deterministic function $g : {\rm support} \, h \subset \A \times \B \ \to \mathbb{R}$, form the total score from these pairs,
say
\begin{equation}\label{def G}
   G \equiv G(m_1,m_2) = \sum_{1 \le i \le m_1,1 \le j \le m_2}  g(A_i,B_j) \, h(A_i,B_j).
\end{equation}

The natural estimator is the observed average score per pair,  $G/W$.  Unfortunately, this is not \emph{unbiased}.
However, it is \emph{consistent}, 
which is \eqref{claim GW} in the following theorem.
 
\begin{theorem}\label{unbiased}
For bounded $g$, 
$G/(m_1m_2)$ gives an unbiased estimator for $p \,  \e g(S)$.
 Furthermore, as $m_1,m_2 \to \infty$,
 \begin{equation}\label{claim G}
G(m_1,m_2)/(m_1m_2) \to p \,  \e g(S), 
\end{equation} 
and
\begin{equation}\label{claim GW}
G(m_1,m_2)/W(m_1,m_2) \to  \e g(S), 
\end{equation}
where the convergence is convergence in probability.
\end{theorem}
\begin{proof}
To show unbiased:  for each $i,j$ we have, since $h$ is an indicator, with $\e h(A_i,B_j)=p$,
$$
  \e  g(A_i,B_j) \, h(A_i,B_j)  = \e ( g(A_i,B_j) \, |  h(A_i,B_j)=1) \ \times p . 
$$
According to \eqref{def S}, the conditional distribution of $(A_i,B_j)$ given $h=1$ is equal to the distribution of $S$, so the right-hand side of the display above equals $p \,  \e g(S)$.
  
To prove \eqref{claim G}, start by writing $X_{ij}= g(A_i,B_j) \, h(A_i,B_j)$.  In the variance-covariance expansion of Var($G)$, there are $m_1m_2$ terms \\ cov($X_{ij},X_{i'j'})$ where both $i=i'$ and $j=j'$, $m_1m_2^2$ terms
where only $i=i'$, $m_1^2m_2$ terms where only $j=j'$, and all other terms --- involving four independent random variables
$A_i,A_{i'},B_j,B_{j'}$ --- are zero.  Hence Var($G/(m_1m_2))$ $ \to 0$, and Chebyshev's inequality implies the desired convergence in probability. 

Taking the special case $g=1$, the random variable $G$ is the count $W$,  so \eqref{claim G}
implies that
$W/(m_1m_2)$ $ \to p$ in probability.  Combined with \eqref{claim G} for the general $g$, this
proves \eqref{claim GW}.
\end{proof}
 }

\section{Computational considerations for implementing PDC on integer partitions}
\label{appendix}


In Section \ref{Algorithms Section}, several methods were presented for the simulation of partitions of the integer $n$.  The analysis focused on the asymptotic rejection probabilities, which varied by choice of $(A,B)$.   

The punchline is: instead of rejection sampling (Algorithm~\ref{IP WTGL procedure}), the default algorithm should always be PDC deterministic second half with von Neumann's rejection method (Algorithm~\ref{IP PDC DSH procedure}).  It is guaranteed faster than rejection sampling, requires the same amount of memory, and under various restrictions on the parts of an integer partition it is still applicable even when self-similar bijections do not exist.  In addition, it is no more difficult to program on a computer.

\emph{Then}, if more efficiency is needed, one can weigh the costs of storing all or part of a table of values or exploring the existence of self-similar bijections.  



\ignore{
 \subsection{Method of choice for unrestricted partitions}\label{sect choice}

If one is interested in generating just a few partitions of a moderately sized $n$, then the waiting-to-get-lucky method, with a 
 ``time $n$'' method of proposal for $(Z_1,\ldots,Z_n)$,  is very easy to program.  The overall runtime is order of $n^{7/4}$ --- a factor of $n$ to make a proposal,\footnote{A smarter  ``time $\sqrt{n}$\,'' method of proposal for $(Z_1,\ldots,Z_n)$ is described  in Secton \ref{sect cost}.  It is harder to program, but gets the overall runtime down to order of $n^{5/4}$.} and a factor of $n^{3/4}$ for waiting-to-get-lucky.  
For example, the  Matlab\textsuperscript{\textregistered} \cite{MATLAB} code\footnote{Line \eqref{naive matlab} implements the proposal, and is explained in Section \ref{sect naive}; the line before and the line after \eqref{naive matlab} implement waiting-to-get-lucky.}
{\tt
\begin{eqnarray}
\nonumber&& \text{n=100; logxn=-sqrt(6*n)/pi; s=0;}   \\
\nonumber&& \text{while s\~{}=n,} \\ 
\label{naive matlab}&&    \text{Z=floor(log(rand(n,1))./(1:n)'.*logxn);}\\  
\nonumber&&    \text{s=(1:n)*Z;end} 
\end{eqnarray}}
\hspace{-.77em}   runs on a common desktop computer\footnote{Macintosh iMac, 3.06 Ghz, 4 GB RAM.} 
at around 600 partitions of $n=100$ per second; with $n=1000$ the same runs at about 20 partitions per second, and at $n=10,000$ takes about 2 seconds per partition.

\ignore{ The supporting evidence is
tic,rand('state',0),for rr=1:100,n=1000;logx=-sqrt(6*n)/pi;s=0;while s~=n,Z=floor(log(rand(n,1))./(1:n)'.*logx);s=(1:n)*Z;end,end,toc
Elapsed time is 4.244796 seconds.
>> tic,rand('state',0),for rr=1:100,n=100;logx=-sqrt(6*n)/pi;s=0;while s~=n,Z=floor(log(rand(n,1))./(1:n)'.*logx);s=(1:n)*Z;end,end,toc
Elapsed time is 0.158214 seconds.
tic,rand('state',0),for rr=1:100,n=10000;logx=-sqrt(6*n)/pi;s=0;while s~=n,Z=floor(log(rand(n,1))./(1:n)'.*logx);s=(1:n)*Z;end,end,toc
Elapsed time is 207.702194 seconds.
}  

The table method is by far the fastest method, if
one is interested in generating many samples, and the table of size $n^2$ floating point numbers fits in random access memory.
For example, for $n=10,000$ the same computer as above takes 5 seconds to generate the table --- a one time cost, and then finds 40 partitions per second.
At $n=15,000$, the same computer takes 28 seconds to generate the table, and then finds 25 partitions per second.   But at $n= 19,000$, the computer freezes, as too much memory was requested.

\ignore{ The supporting evidence is
>> tic,n=10000;kk=n;x=exp(-pi/sqrt(6*n));makexPartitionTable,toc
Elapsed time is 4.807604 seconds.
>> tic, for rr=1:1000, getpartition,end;toc
Elapsed time is 25.816201 seconds.
>> tic,n=10000;kk=n;x=exp(-pi/sqrt(6*n));makexPartitionTable,toc
Elapsed time is 4.555448 seconds.
>> tic,n=11000;kk=n;x=exp(-pi/sqrt(6*n));makexPartitionTable,toc
Elapsed time is 11.256547 seconds.
>> clear,tic,n=11000;kk=n;x=exp(-pi/sqrt(6*n));makexPartitionTable,toc
Elapsed time is 6.059108 seconds.
>> clear,tic,n=12000;kk=n;x=exp(-pi/sqrt(6*n));makexPartitionTable,toc
Elapsed time is 15.343318 seconds.
clear,tic,n=13000;kk=n;x=exp(-pi/sqrt(6*n));makexPartitionTable,toc
Elapsed time is 10.212148 seconds.
clear,tic,n=14000;kk=n;x=exp(-pi/sqrt(6*n));makexPartitionTable,toc
Elapsed time is 16.258301 seconds.
>> tic, for rr=1:1000, getpartition,end;toc
Elapsed time is 37.649698 seconds.
clear,tic,n=15000;kk=n;x=exp(-pi/sqrt(6*n));makexPartitionTable,toc
Elapsed time is 27.851075 seconds.
>> tic, for rr=1:1000, getpartition,end;toc
Elapsed time is 40.713551 seconds.
 clear,tic,n=16000;kk=n;x=exp(-pi/sqrt(6*n));makexPartitionTable,toc
Elapsed time is 42.583734 seconds.
clear,tic,n=17000;kk=n;x=exp(-pi/sqrt(6*n));makexPartitionTable,toc
Elapsed time is 44.105427 seconds.
 clear,tic,n=18000;kk=n;x=exp(-pi/sqrt(6*n));makexPartitionTable,toc
Elapsed time is 54.846707 seconds.
 clear,tic,n=18000;kk=n;x=exp(-pi/sqrt(6*n));makexPartitionTable,toc
Elapsed time is 54.676892 seconds.
>> tic, for rr=1:1000, getpartition,end;toc
Elapsed time is 54.248130 seconds.

  USING office computer
>>	
Elapsed time is 25.891377 seconds.
>> tic,for rr=1:2,n=2^20;getZZ,end,toc
Elapsed time is 2.427617 seconds.
>> tic,for rr=1:25,n=2^20;getZZ,end,toc
Elapsed time is 24.284738 seconds.
>> tic,for rr=1:25,n=2^24;getZZ,end,toc
Elapsed time is 59.996582 seconds.
>> tic,for rr=1:25,n=2^28;getZZ,end,toc
Elapsed time is 142.021811 seconds.
>> tic,for rr=1:25,n=2^30;getZZ,end,toc
Elapsed time is 364.425667 seconds.

using home computer
nrt=[];for ee=18:2:32,n=2^ee;tic,for rr=1:100,getZZ,end,nrt=[nrt;n,rr,toc],end
nrt =
                    262144                       100              96.748373929
                   1048576                       100             130.860450022
                   4194304                       100             162.146577866
                  16777216                       100              202.51236085
                  67108864                       100             283.496338926
                 268435456                       100             469.360962572
}

The divide-and-conquer methods of Sections \ref{sect divide} and \ref{sect b=1}, using the small versus large division of
\eqref{b}, 
offer a large speedup over waiting-to-get-lucky, but only for  case
$b=1$, with its deterministic second half, can we analyze the speedup --- the $\sqrt{n}/c$ factor in Theorem \ref{b=1 theorem}.

The divide-and-conquer method based on $p(z) =d(z) p(z^2) $ is unbeatable for large $n$.  Regardless of the manner of costing,
be it only counting random bits used, or uniform costing, or logarithmic (bitop) costing, the cost to find a random partition of $n$ must be asymptotically at least as large as the time to propose $(Z_1,\ldots,Z_n)$ for a random partition of a random number around $n$.   The entire divide-and-conquer algorithm of Theorem \ref{theorem sum of series variation}, compared with just proposing 
$(Z_1,\ldots,Z_n)$, has asymptotically an extra cost factor of $\sqrt{2}$. So the claim of \emph{unbeatable} at the start of this paragraph really means:  asymptotically unbeatable by anything more than a factor of $ \sqrt{2}$.

}

\subsection{ On proposing an instance of the independent process}
All of the algorithms for integer partitions center around the generation of variates from the process of independent random variables $\UZ =(Z_1, Z_2, \ldots, Z_n)$, sampled under \eqref{dist Z} and \eqref{def x(n)}.
The number of random bits needed to sample from this process is at least the (base 2)  entropy $H(\UZ)$.
With $x$ chosen as in \eqref{def x(n)}, it is 
nontrivial\footnote{Using the notation of entropy and conditional entropy as in \cite{cover}, 
\begin{align*}
H(\UZ) & = \sum_m \p(T=m) H(\UZ | T=m)+ H(T) \\
 & \geq \sum_m \p(T=m) H(\UZ |T=m) \\
 & = \sum_m \p(T=m) \log_2(p(m)).
 \end{align*}
Then using Chebyshev's inequality together with \eqref{SD}, we see that the sum can be restricted to $m \ge m_0 =(1-\varepsilon) n$ with $\p(T \ge m_0) \to 1$ and $\varepsilon \to 0$, which proves the lower bound.  The upper bound follows by observing that $H(T) = o(\log(p((1-\varepsilon)n)))$ for all $\varepsilon$ sufficiently small.}
to see that this entropy is asymptotically $\log_2(p(n))$, and hence, by Hardy--Ramanujan \eqref{HR}, asymptotic to
$(2/\ln 2) c \sqrt{n}$, with $c=\pi/\sqrt{6}$.  

\subsubsection{Na\"{\i}ve proposals}

The simplest procedure for sampling from the independent process $\UZ =(Z_1, Z_2, \ldots, Z_n)$ is to sample each coordinate separately, using the fact that, if $Z$ is geometrically distributed with parameter $p$, then $Z =_d \lfloor \ln(U) / \ln(1-p)\rfloor$ for $U$ uniform over the interval $(0,1)$.  Iterating through all coordinates supplies a $O(n)$ procedure for sampling $\UZ$. 

\ignore{, given in Algorithm~\ref{Naive Proposal}.

\begin{algorithm}{\rm
\begin{algorithmic}
\State Let $u_1, u_2, \ldots u_n$ denote i.i.d.~uniform $(0,1)$ random variables.
\State Let $x = e^{-\pi/\sqrt{6n}}$.  
\For {$i=1,...,n$}
\State $z_i = \log(u_i) / (i \log(x) )$
\EndFor
\State \Return $( z_1, \ldots, z_n )$.  
\end{algorithmic}}
\caption{Naive Proposal for $\UZ = (Z_1, \ldots, Z_n)$} 
\label{Naive Proposal}
\end{algorithm}
} 
Erd\H{o}s and Lehner \cite{lehner} showed that with probability tending to 1,  the largest part  is close to  $2 c \sqrt{n} \log(n)$, and that the number of part-sizes, corresponding to  the number of nonzero $Z_i$'s in $\UZ$, is close to $(1/c) \sqrt{n}$, with $c = \pi/\sqrt{6}$.  
This implies that with high probability $Z_i = 0$ for all $i\gg \sqrt{n} \log(n)$.  
There are several adaptations one can make, such as pooling the variables together under a single uniform random variable, but the implementations are messy.

 A variation due to Sheldon Ross\footnote{Personal communication.} would generate $L$, the largest index for which $Z_j > 0$, whose distribution is given by
$$
\p(L = j) = x^j \prod_{k>j} (1-x^k).
$$
Conditional on $\{L=j\}$ for some $j>0$, the distribution of $\UZ$ is equal to the distribution of the vector  $(Z_1, Z_2, \ldots, 1+Z_j,0,0,\ldots)$.  

\ignore{ 
\subsubsection{A na\"{\i}ve proposal, with an extra cost factor $\sqrt{ n}$}
\label{sect naive}
 The algorithm \eqref{naive matlab} of Section \ref{sect choice} generates a geometric random variable $Z$ with parameter $a \in (0,1)$, so that $\p(Z \ge k)=a^k$ for $k=0,1,2,\ldots$, via
 $Z = \lfloor \,\ln U  / \ln a \rfloor$, where $U$ is uniformly distributed in $(0,1)$.  That $Z$ is geometrically distributed follows from the inversion method, see for example \cite{DeVroye}, or more directly, by the following calculation: $\p(  \lfloor \,\ln U  / \ln a \rfloor \ge k) = \p( \ln U  / \ln a \ge k) = \p( U \le a^k) =  a^k$.
This is applied with $a = x^i = \exp(-ic/\sqrt{n})$, for $i=1$ to $n$. 
This approach, generating each $Z_i$, $i=1$ to $n$, using $n$ independent random uniforms, 
results in a ``time $n$" procedure 
to propose one instance of $\UZ$.

\subsubsection{A proposal, with an extra cost factor $\log n$}

Erd\H{o}s and Lehner \cite{lehner} showed that with probability tending to 1,  the largest part  is close to  $2 c \sqrt{n} \log(n)$, and that the number of part-sizes, corresponding to  the number of nonzero $Z_i$'s in $\UZ$, is close to $(1/c) \sqrt{n}$.  
This implies that with high probability $Z_i = 0$ for all $i\gg \sqrt{n} \log(n)$.  
One way to exploit this, which we call a `collective coin toss' method\footnote{This is also similar to pooled sample data, see for example  \cite[Chapter IX Problem 26]{Feller}.}, 
 is to pick $k_n$ a little bigger than $2 c \sqrt{n} \log(n)$
and compute
$$
\p(Z_i = 0 \text{ for all } i \geq k_n) =  \prod_{i=k_n}^n ( 1-x^i) =: 1-\epsilon_{k_n}.
$$
If the first random uniform $U$ gets lucky and falls within the interval $[\epsilon_{k_n},1]$, then we have already determined that $Z_i=0$ for $i=k_n$ to $n$.  Of course, fraction $\epsilon_{k_n}$ of the trials fail to be lucky,  complicating the proposal procedure.   
 
 A variation due to Sheldon Ross\footnote{private communication.} would generate $L$, the largest index for which $Z_j > 0$, whose distribution is given by
$$
\p(L = j) = x^j \prod_{k>j} (1-x^k).
$$
Conditional on $(L=j)$ for some $j>0$, the distribution of $\UZ$ is equal to the distribution of the vector  $(Z_1, Z_2, \ldots, 1+Z_j,0,0,\ldots)$.  
}

\subsubsection{A proposal on the same order as the lower entropy bound}
\label{sect soph}
Our recommended proposal is one that takes advantage of the relation between geometric and Poisson random variables.  See \cite{bodini} for an alternative description.

A geometric random variable $Z$ with parameter $0<a<1$ can be represented as a sum of independent Poisson random variables $Y_j$, $j=1,2,\ldots$, as $Z = \sum_j j Y_j$, where $\e Y_j = a^j/j$ (this is easily verified using generating functions).  The random variables $Y_j$ can be generated via a Poisson process as follows.  Let $r = \sum_j a^j/j$, and divide the interval $[0,r]$ into disjoint intervals of length $a, a^2/2, a^3/3$, etc.  Then $Y_j = k$ if exactly $k$ arrivals occur in the interval of length $a^j/j$.

To simulate the vector $(Z_1(x),Z_2(x),\ldots,Z_i(x),\ldots)$, with parameters $a=x, x^2, \ldots,x^i,\ldots$, we fix disjoint intervals of length $x^{ij}/j$ for $i,j \ge 1$, and run a Poisson process on the interval $[0,s]$, where $s = \sum_{i,j} x^{ij}/j$.

Our claim that we have an algorithm using $O(\sqrt{n})$ calls to a random number generator is supported by the calculation here that, with $x=x(n)=\exp(-c/\sqrt{n})$ and $c=\pi/\sqrt{6}$, 
\begin{equation}\label{def s(n)}
s(n):= \sum_{i,j \ge 1} \frac{x^{ij}}{j} = \sum_j \frac{1}{j^{2}} \frac{ jx^j}{1-x^j} \le \frac{\pi^2}{6} \, \frac{x}{1-x} \sim c \sqrt{n};
\end{equation}
the inequality follows from the observation that for $0<x<1$, $jx^j \leq x+x^2+\ldots+x^j \leq x(1-x_j)/(1-x)$.

\subsection{Floating point considerations and coin tossing}
\label{floating}
Is floating point accuracy sufficient, in the context of computing an acceptance threshold $t(a)$?  There is a very concrete answer, based on \cite{KnuthYao}, see \cite[ Section 5.12]{cover} for an accessible exposition.
First, given $p \in (0,1)$, a $p$-coin can be tossed using a random number of fair coin tosses;  the expected number is exactly 2, unless $p$ is a $k$th level dyadic rational, i.e., $p=i/2^k$ with odd $i$, in which case the expected number is $2-2^{1-k}$.  The proof is by consideration  of say $B,B_1,B_2,\ldots$ i.i.d.~with $\p(B=0)=\p(B=1)=1/2$;  after $r$ tosses we have determined the first $r$ bits of the binary expansion of a random number $U$ which is uniformly distributed in $(0,1)$, and the usual simulation recipe  is that a $p$-coin is the indicator  $1(U <p)$.  Unless $\lfloor 2^r\, p \rfloor
= \lfloor 2^r \, U \rfloor$, the first $r$ fair coin tosses will have determined the value of the $p$-coin.
Exchanging the roles of $U$ and $p$, we see that the number of bits of precision read off of $p$ is, on average, 2, and exceeds 
$r$ with probability $2^{-r}$.
If a  floating point number delivers  50 bits of precision, the chance of needing more precision is $2^{-50}$ per
evaluation of an indicator of the form $1(U<p)$.  Our divide-and-conquer doesn't require very many acceptance/rejection decisions;  for example, with $n=2^{60}$, there are about 30 iterations of the algorithm in Theorem
\ref{prop root 2}, each involving on average about $\sqrt{2}$ acceptance/rejection decisions, according to Theorem
\ref{recursive theorem}.  So one might program the algorithm to deliver exact results; most of the time  determining  acceptance thresholds $p=t(a)$ in  floating point arithmetic, but keeping  track of whether more  bits of $p$ are needed.  On the  unlikely event, of probability around $30 \times  \sqrt{2} /2^{50} < 4 \times 10^{-14}$, that more precision is needed, the program demands 
a more accurate calculation of $t(a)$.
  This would be far more efficient than using extended integer arithmetic to calculate values of $p(n)$ exactly; see Remark~\ref{X} and Section~\ref{sect cost}.

Another place to consider the use of floating point arithmetic is in proposing the vector $(Z_1(x),\ldots,Z_n(x))$. 
If one call to the random number generator suffices to find the next arrival in a rate 1 Poisson process, we have an algorithm using $O(\sqrt{n})$ calls, which can propose the entire vector $(Z_1,Z_2,\ldots)$, using $x=x(n)$ from \eqref{def x(n)}.
The proposal algorithm, summarized in Section \ref{sect soph}, is  based on a compound Poisson representation of geometric distributions, and is similar to a coupling used in \cite{budalect}, Section 3.4.1.

Once again, suppose we want to guarantee \emph{exact} simulation of a proposal $(Z_1(x),\ldots,Z_n(x))$.
In the standard rate 1 Poisson process, run up to time $s(n)$, given by \eqref{def s(n)}, we need to assign \emph{exactly}, for each arrival, say at a random time $R$, the corresponding index
$(i,j)$, such that the partial sum for $s(n)$ up to, but excluding the $ij$ term, is less than $R$, but the partial sum, including the $ij$ term, is greater than or equal to $R$.  Based on an entropy result from Knuth and Yao \cite{KnuthYao}, a crucial quantity is
\begin{equation}\label{knuth h}
h(n):= \sum_{i,j \ge 1} \frac{x^{ij}}{j}  \ln \frac{j}{x^{ij}}  \le (c -\zeta'(2)/c) \ \sqrt{n}.
\end{equation}
An exact simulation of the Poisson process, assigning $ij$ labels to each arrival, 
can be done\footnote{on each interval $(m-1,m]$ for $m=1$ to $\lceil s(n) \rceil$, perform an exact simulation of the number of arrivals, which is distributed according to the Poisson distribution with mean 1.  For each arrival on $(m-1,m]$, there is a discrete distribution described by those partition points lying in $(m-1,m)$, together with the endpoints $m-1$ and $m$; calling the corresponding random variable $X_m$, the sum of the base 2 entropies satisfies $\sum h(X_m) \le h(n)+ s(n)$,  since each extra subdivision of one of the original subintervals of
length $x^{ij}/j$ adds at most one bit of entropy. }
 with $O( s(n)+h(n))$ genuine random bits, and the bounds for $s(n)$ and $h(n)$ show that this is $O(\sqrt{n})$.

\ignore{The costing scheme which counts only the expected number of random bits needed is clearly inadequate.  Consider the impractical algorithm:  list all $p(n)$ partitions, for example in lexicographic order.  Use $\lceil \log_2 p(n) \rceil$ random bits to choose an integer $I$ uniformly from
$\{1,2\ldots,p(n)\}$.  Report the $I^{\rm th}$ partition of $n$.      If one costs only by the number of random bits needed, the algorithm just described is \emph{strictly} unbeatable! 
 }
 
\section{Acknowledgements}

The authors would like to thank the anonymous referees for comments that contributed to the improvement of this paper, and to Fredrik Johansson for help with the proof of Lemma~\ref{r bits}. 

SD was partially supported by a Dana and David Dornsife final year dissertation fellowship.

 \bibliographystyle{plain}
\bibliography{PDC}

 \end{document}